\documentclass[12pt]{amsart}
\usepackage{amsmath,amsfonts,amssymb,amsthm,amsbsy,amscd,latexsym,cite}
\usepackage[cp1251]{inputenc}
\usepackage[colorlinks,breaklinks,linkcolor=blue,citecolor=blue, menucolor=blue,pagecolor=blue]{hyperref}

\multlinegap=\parindent

\begin{document}

\title[Certain residual properties of~HNN-extensions]{Certain residual properties of~HNN-extensions with~central associated subgroups}

\author{E.~V.~Sokolov}
\address{Ivanovo State University, Russia}
\email{ev-sokolov@yandex.ru}

\begin{abstract}
Suppose that $G$ is a~group, $H$~and~$K$ are proper isomorphic central sub\-groups of~$G$, and~$\mathfrak{G}$ is an~HNN-ex\-ten\-sion of~$G$ with~the~associated subgroups~$H$ and~$K$. We prove necessary and~sufficient conditions for~$\mathfrak{G}$ to~be residually a~$\mathcal{C}$\nobreakdash-group, where $\mathcal{C}$ is a~class of~groups closed under taking subgroups, extensions, homomorphic images, and~Cartesian products of~the~form $\prod_{y \in Y}X_{y}$, where $X, Y \in \mathcal{C}$ and~$X_{y}$ is an~isomorphic copy of~$X$ for~each $y \in Y$.
\end{abstract}


\maketitle

\newtheorem{etheorem}{Theorem}
\newtheorem{ecorollary}{Corollary}
\newtheorem{eproposition}{Proposition}[section]

\vspace{-7pt}

\section{Introduction. Statement of~results}\label{esec01}

Let $\mathcal{C}$ be a~class of~groups. Recall that a~group~$X$ is said to~be \emph{residually a~$\mathcal{C}$\nobreakdash-group} if, for~any non-triv\-i\-al element $x \in X$, there is a~homomorphism~$\sigma$ of~$X$ onto~a~group from~$\mathcal{C}$ such that $x\sigma \ne 1$.

The~main question in~the~study of~the~$\mathcal{C}$\nobreakdash-residu\-al\-ity of~HNN-ex\-ten\-sions is whether an~HNN-ex\-ten\-sion of~a~residually $\mathcal{C}$\nobreakdash-group is again residually a~$\mathcal{C}$\nobreakdash-group. All known answers to~this question are obtained under various restrictions imposed on~the~base group, the~associated subgroups, and~/~or~the~isomorphism between them (the~terminology for~HNN-ex\-ten\-sions used here and~below follows~\cite{LyndonSchupp1977}). One such restriction is the~centrality of~the~associated subgroups in~the~base group; the~$\mathcal{C}$\nobreakdash-residu\-al\-ity of~HNN-ex\-ten\-sions of~this type (including the~case of~an~abelian base group) is studied in~\cite{RaptisVarsos1987, AndreadakisRaptisVarsos1988, RaptisVarsos1989, RaptisVarsos1991, KimTang1999, WongGan1999, Moldavanskii2002, Moldavanskii2003, WongWong2005, Goltsov2015, SokolovTumanova2017, SokolovTumanova2019}.

In~\cite{AndreadakisRaptisVarsos1988, RaptisVarsos1989, RaptisVarsos1991}, S.~Andreadakis, E.~Raptis, and~D.~Varsos give criteria for~the~residual finiteness and~the~residual nilpotence of~an~HNN-ex\-ten\-sion of~a~finitely generated abelian group and~prove that such an~extension is residually solvable. D.~I.~Moldavanskii~\cite{Moldavanskii2002, Moldavanskii2003} significantly strengthens some of~these results by~generalizing them to~the~case of~HNN-ex\-ten\-sion with~central associated subgroups. He proposes an~original approach to~the~study of~the~$\mathcal{C}$\nobreakdash-residu\-al\-ity of~such HNN-ex\-ten\-sions, which he call \emph{the~method of~descent and~ascent of~compatible subgroups}. In~\cite{Moldavanskii2002, Moldavanskii2003}, this method is applied under the~assumption that $\mathcal{C}$ is the~class of~all finite groups or~all finite $p$\nobreakdash-groups, where $p$ is a~prime number. The~aim of~this paper is to~generalize the~results obtained in~\cite{Moldavanskii2002, Moldavanskii2003} to~the~case when $\mathcal{C}$ is an~arbitrary root class of~groups closed under taking quotient groups.

The~notion of~a~root class was introduced by~K.~Gruenberg~\cite{Gruenberg1957}, and~its equivalent definitions are given in~\cite{Sokolov2015}. In~accordance with~one of~them, a~class of~groups~$\mathcal{C}$ containing at~least one non-triv\-i\-al group is called \emph{root} if it is closed under taking subgroups, extensions, and~Cartesian products of~the~form $\prod_{y \in Y}X_{y}$, where $X, Y \in \mathcal{C}$ and~$X_{y}$ is an~isomorphic copy of~$X$ for~each $y \in Y$. The~examples of~root classes are the~above-mentioned classes of~all finite groups and~all finite $p$\nobreakdash-groups, as~well as~the~classes of~periodic $\pi$\nobreakdash-groups of~finite exponent (where $\pi$ is a~non-empty set of~primes), all solvable groups, and~all tor\-sion-free groups. We note also that the~intersection of~any number of~root classes is again a~root class~\cite{Sokolov2015}. 

The~notion of~a~root class turns~out to~be very useful in~studying the~residual properties of~free constructions of~groups allowing one to~prove many statements at~once using the~same argument. The~papers~\cite{Gruenberg1957, AzarovTieudjo2002} are the~starting point for~these studies, the~latest results obtained in~this area can be found in~\cite{Tumanova2019, SokolovTumanova2019, SokolovTumanova2020SMJ, SokolovTumanova2020IVM, SokolovTumanova2020LJM, Tumanova2020, Sokolov2021JA, Sokolov2021SMJ1, Sokolov2021SMJ2}. The~root-class residuality of~HNN-ex\-ten\-sions is studied in~\cite{Tieudjo2010, Tumanova2014, Goltsov2015, Tumanova2017, SokolovTumanova2017, SokolovTumanova2019, Tumanova2020}. These papers deal mainly with~the~cases when the~associated subgroups coincide or~intersect trivially.

Let us call a~sequence $(X, Y, Z, \psi)$ an~\emph{HNN-tuple} if $X$ is a~group, $Y$ and~$Z$ are isomorphic subgroups of~$X$, and~$\psi\colon Y \to Z$ is an~isomorphism. If $(X, Y, Z, \psi)$ is an~HNN-tuple, then by~$\mathrm{HNN}(X, Y, Z, \psi)$ we denote the~HNN-ex\-ten\-sion $\langle X, t;\ t^{-1}Yt = Z, \psi \rangle$. Recall that $\mathrm{HNN}(X, Y, Z, \psi)$ is the~group whose generators are the~generators of~$X$ and~the~symbol~$t$, and~whose defining relations are the~relations of~$X$ and~all possible relations of~the~form $t^{-1}yt = y\psi$, where~$y$ and~$y\psi$ are words in~the~generators of~$X$ defining an~element $y \in Y$ and~its image under~$\psi$.

Throughout this section, we assume that $(G, H, K, \varphi)$ is an~HNN-tuple, $H$~and~$K$ \underline{lie in~the~center} of~$G$, and~$\mathfrak{G} = \mathrm{HNN}(G, H, K, \varphi)$. We put $K_{0} = G$, $H_{1} = H$, $K_{1} = K$, and~(if $H_{i}$ and~$K_{i}$ are already defined) $H_{i+1} = H_{i} \cap K_{i}$, $K_{i+1} = H_{i+1}\varphi$. To~simplify the~notation, the~restriction of~$\varphi$ to~$H_{i}$ ($i \geqslant 1$) or~some other subgroup is denoted below by~the~same symbol~$\varphi$.

Since, for~any $i \geqslant 0$, the~sequence $(K_{i}, H_{i+1}, K_{i+1}, \varphi)$ is an~HNN-tuple, the~group\linebreak $\mathfrak{K}_{i} = \mathrm{HNN}(K_{i}, H_{i+1}, K_{i+1}, \varphi)$ is defined. Obviously, if $H_{n} = K_{n}$ for~some $n \geqslant 1$, then $H_{n+1} = K_{n} = K_{n+1}$, and~therefore $\mathfrak{K}_{n}$ is a~split extension of~$K_{n}$ by~the~infinite cyclic group~$\langle t \rangle$. It~is also easy to~see that the~restriction of~$\varphi$ to~$H_{n}$ turns~out to~be an~automorphism of~this subgroup, and~$\mathfrak{K}_{n}$ is isomorphic to~the~subgroup $E = \operatorname{sgp}\{H_{n}, t\}$ of~$G$, which is a~split extension of~$H_{n}$ by~$\langle t \rangle$.

The~method of~descent and~ascent of~compatible subgroups essentially consists in~proving that, under certain conditions, for~each $i \geqslant 0$, the~$\mathcal{C}$\nobreakdash-residu\-al\-ity of~$G$ is equivalent to~the~$\mathcal{C}$\nobreakdash-residu\-al\-ity of~$\mathfrak{K}_{i}$. If $H_{n} = K_{n}$ for~some $n \geqslant 1$, this allows us to~reduce the~question of~the~$\mathcal{C}$\nobreakdash-residu\-al\-ity of~$G$ to~the~much simpler problem of~finding conditions for~the~split extension~$\mathfrak{K}_{n}$ to~be residually a~$\mathcal{C}$\nobreakdash-group. Two criteria for~the~root-class residuality of~split extensions are given at~the~end of~Section~\ref{esec05}.

We note that the~equality $H_{n} = K_{n}$ may not hold for~any~$n$. To~guarantee its fulfillment, the~theorems given below impose a~weaker condition on~$G$: $H_{n} = H_{n+1}$ for~some~$n$. Corollaries~\ref{ec01}\nobreakdash--\ref{ec04} describe a~number of~situations in~which the~last relation certainly takes place.

Throughout the~paper, if $\mathcal{C}$ is a~class of~groups and~$X$ is a~group, then $\mathcal{C}^{*}(X)$ denotes the~family of~normal subgroups of~$X$ such that $Y \in \mathcal{C}^{*}(X)$ whenever $X/Y \in \mathcal{C}$. If~$\mathcal{C}$~consists only of~periodic groups, then we denote by~$\pi(\mathcal{C})$ the~set of~all the~prime numbers that divide the~orders of~the~elements of~all possible $\mathcal{C}$\nobreakdash-groups. Recall that if $\pi$ is a~set of~primes, then a~\emph{$\pi$\nobreakdash-num\-ber} is an~integer all of~whose prime divisors belong to~$\pi$, and~a~\emph{$\pi$\nobreakdash-group} is a~periodic group in~which the~order of~each element is a~$\pi$\nobreakdash-num\-ber.

\begin{etheorem}\label{et01}
Suppose that $\mathcal{C}$ is a~root class of~groups closed under taking quotient groups, $G$ is residually a~$\mathcal{C}$\nobreakdash-group, and~there exists a~subgroup $Q \in \mathcal{C}^{*}(G)$ satisfying at~least one of~the~following conditions: 

$(\alpha)$\hspace{1ex}$H \cap Q = 1 = K \cap Q$,

$(\beta)$\hspace{1ex}$Q \leqslant H \cap K$ and~$Q\varphi = Q$.

\textup{I.\phantom{I}}\hspace{1ex}If $\mathcal{C}$ contains non-pe\-ri\-od\-ic groups, then $\mathfrak{G}$ is residually a~$\mathcal{C}$\nobreakdash-group.

\textup{II.}\hspace{1ex}If $\mathcal{C}$ consists only of~periodic groups, $H \ne G \ne K$, and~$H_{n} = H_{n+1}$ for~some $n \geqslant 1$, then $\mathfrak{G}$ is residually a~$\mathcal{C}$\nobreakdash-group if and~only~if

\textup{1)}\hspace{1ex}$H_{n} = K_{n}$\textup{;}

\textup{2)}\hspace{1ex}the subgroup $E = \operatorname{sgp}\{H_{n}, t\}$ is residually a~$\mathcal{C}$\nobreakdash-group.
\end{etheorem}

\enlargethispage{8pt}

\begin{ecorollary}\label{ec01}
If $\mathcal{C}$ is a~root class of~groups closed under taking quotient groups, $G$ is residually a~$\mathcal{C}$\nobreakdash-group, $H$~and~$K$ are finite, then $H_{n} = K_{n}$ for~some $n \geqslant 1$ and~the~following statements hold.

\textup{I.\phantom{I}}\hspace{1ex}If $\mathcal{C}$ contains non-pe\-ri\-od\-ic groups, then $\mathfrak{G}$ is residually a~$\mathcal{C}$\nobreakdash-group.

\textup{II.}\hspace{1ex}If $\mathcal{C}$ consists only of~periodic groups, then $\mathfrak{G}$ is residually a~$\mathcal{C}$\nobreakdash-group if and~only if the~order of~the~automorphism~$\varphi$ of~$H_{n}$ is a~$\pi(\mathcal{C})$\nobreakdash-num\-ber.
\end{ecorollary}

\begin{ecorollary}\label{ec02}
Suppose that $\mathcal{C}$ is a~root class of~groups closed under taking quotient groups, $H$~and~$K$ have finite index in~$G$, and~$H \ne G \ne K$.

\textup{I.\phantom{I}}\hspace{1ex}If $\mathcal{C}$ contains non-pe\-ri\-od\-ic groups, then $\mathfrak{G}$ is residually a~$\mathcal{C}$\nobreakdash-group if and~only if $G$ is residually a~$\mathcal{C}$\nobreakdash-group and~there exists a~subgroup $Q \in \mathcal{C}^{*}(G)$ satisfying the~conditions $Q \leqslant H \cap K$ and~$Q\varphi = Q$.

\textup{II.}\hspace{1ex}If $\mathcal{C}$ consists only of~finite\kern-23.5pt\rule[-1.92pt]{23.5pt}{.48pt} groups, then $\mathfrak{G}$ is residually a~$\mathcal{C}$\nobreakdash-group if and~only~if

\textup{1)}\hspace{1ex}$G/H \in \mathcal{C}$ and~$G/K \in \mathcal{C}$\textup{;}

\textup{2)}\hspace{1ex}$H_{n} = K_{n}$ for~some $n \geqslant 1$\textup{;}

\textup{3)}\hspace{1ex}the~subgroup $E = \operatorname{sgp}\{H_{n}, t\}$ is residually a~$\mathcal{C}$\nobreakdash-group.
\end{ecorollary}

Let $\pi$ be a~set of~primes. Following~\cite{Sokolov2014}, we call an~abelian group~$A$ \emph{$\pi$\nobreakdash-bound\-ed} if, for~any quotient group~$B$ of~$A$ and~for any $p \in \pi$, the~$p$\nobreakdash-pri\-ma\-ry component of~$B$ is finite. A~nilpotent (solvable) group is said to~be \emph{$\pi$\nobreakdash-bound\-ed} if it has a~finite central (respectively, subnormal) series with~abelian $\pi$\nobreakdash-bound\-ed factors. It is easy to~see that, for~any set~$\pi$ of~prime numbers, a~finitely generated nilpotent group is $\pi$\nobreakdash-bound\-ed nilpotent and~a~polycyclic group is $\pi$\nobreakdash-bound\-ed solvable. It is also known (see Proposition~\ref{ep607} below) that if a~$\pi$\nobreakdash-bound\-ed solvable group is abelian, then it belongs to~the~class of~$\pi$\nobreakdash-bound\-ed abelian groups. Therefore, we can say that an~abelian group is $\pi$\nobreakdash-bound\-ed without specifying the~class of~$\pi$\nobreakdash-bound\-ed groups (abelian, nilpotent, or~solvable) that we mean.{\parfillskip=0pt\par}

Throughout the~paper, we denote by~$\pi^{\prime}$ the~set of~all primes that do not belong to~the~set~$\pi$. Recall that a~subgroup~$Y$ of~a~group~$X$ is said to~be \emph{$\pi^{\prime}$\nobreakdash-iso\-lat\-ed} in~this group if, for~any $x \in X$ and~for any $q \in \pi^{\prime}$, it follows from~the~inclusion $x^{q} \in Y$ that $x \in Y$. Obviously, if $\pi$ contains all primes, then any subgroup turns~out to~be $\pi^{\prime}$\nobreakdash-iso\-lat\-ed. 

It is easy to~see that the~intersection of~any number of~$\pi^{\prime}$\nobreakdash-iso\-lat\-ed subgroups is again a~$\pi^{\prime}$\nobreakdash-iso\-lat\-ed subgroup. Therefore, for~any subgroup $Y \leqslant X$, there exists the~smallest $\pi^{\prime}$\nobreakdash-iso\-lat\-ed subgroup containing it. We call this subgroup the~\emph{$\pi^{\prime}$\nobreakdash-iso\-la\-tor} of~$Y$ in~$X$ and~denote it by~$\mathcal{I}_{\pi^{\prime}}(X, Y)$.

Let $\mathcal{C}$ be a~class of~groups. Recall that a~subgroup~$Y$ is said to~be \emph{$\mathcal{C}$\nobreakdash-sep\-a\-ra\-ble} in~a~group~$X$ if, for~any $x \in X \setminus Y$, there exists a~homomorphism~$\sigma$ of~$X$ onto~a~group from~$\mathcal{C}$ such that $x\sigma \notin Y\sigma$~\cite{Malcev1958}. Obviously, the~$\mathcal{C}$\nobreakdash-residu\-al\-ity of~$X$ is equivalent to~the~$\mathcal{C}$\nobreakdash-sep\-a\-ra\-bil\-ity of~its trivial subgroup. It is also known (see Proposition~\ref{ep603} below) that if $\mathcal{C}$ consists only of~periodic groups, then any $\mathcal{C}$\nobreakdash-sep\-a\-ra\-ble subgroup of~$X$ is $\pi(\mathcal{C})^{\prime}$\nobreakdash-iso\-lat\-ed in~this group. Thus, for~such a~class of~groups~$\mathcal{C}$, the~main problem in~the~study of~$\mathcal{C}$\nobreakdash-sep\-a\-ra\-bil\-ity is the~search for~conditions under which a~$\pi(\mathcal{C})^{\prime}$\nobreakdash-iso\-lat\-ed subgroup turns~out to~be $\mathcal{C}$\nobreakdash-sep\-a\-ra\-ble.

\begin{etheorem}\label{et02}
Suppose that $\mathcal{C}$ is a~root class of~groups consisting only of~periodic groups and~closed under taking quotient groups,\kern-1pt{} $G$\kern-1.5pt{} is residually a~$\mathcal{C}$\nobreakdash-group,\kern-1pt{} $H$\kern-2.5pt{}~and~$K$\kern-2.5pt{} are $\pi(\mathcal{C})$\nobreakdash-bound\-ed, and~$H \ne G \ne K$. Suppose also that $H_{n} = H_{n+1}$ for~some $n \geqslant 1$ and,~for~any $i \in \{0,\, 1,\, \ldots,\, n-1\}$ and~$N \in \mathcal{C}^{*}(H_{i+1}K_{i+1})$, the~subgroup $\mathcal{I}_{\pi(\mathcal{C})^{\prime}}(K_{i}, N)$ is $\mathcal{C}$\nobreakdash-sep\-a\-ra\-ble in~$K_{i}$. Then $\mathfrak{G}$ is residually a~$\mathcal{C}$\nobreakdash-group if and~only~if

\textup{1)}\hspace{1ex}$H_{n} = K_{n}$\textup{;}

\textup{2)}\hspace{1ex}the~subgroup $E = \operatorname{sgp}\{H_{n}, t\}$ is residually a~$\mathcal{C}$\nobreakdash-group\textup{;}

\textup{3)}\hspace{1ex}$H$~and~$K$ are $\pi(\mathcal{C})^{\prime}$\nobreakdash-iso\-lat\-ed in~$G$.
\end{etheorem}

\begin{ecorollary}\label{ec03}
Suppose that $\mathcal{C}$ is a~root class of~groups consisting only of~periodic groups and~closed under taking quotient groups, $G$ is $\pi(\mathcal{C})$\nobreakdash-bound\-ed nilpotent, $H \ne G \ne K$. Suppose also that there exists a~number $m \geqslant 0$ such that $H_{m+1}$ and~$K_{m+1}$ are finitely generated or~$\pi^{\prime}$\nobreakdash-iso\-lat\-ed in~$K_{m}$ for~some finite subset~$\pi$ of~$\pi(\mathcal{C})$. Then $\mathfrak{G}$ is residually a~$\mathcal{C}$\nobreakdash-group if and~only~if

\textup{1)}\hspace{1ex}$H_{n} = K_{n}$ for~some $n \geqslant 1$\textup{;}

\textup{2)}\hspace{1ex}the~subgroup $E = \operatorname{sgp}\{H_{n}, t\}$ is residually a~$\mathcal{C}$\nobreakdash-group\textup{;}

\textup{3)}\hspace{1ex}$G$ has no $\pi(\mathcal{C})^{\prime}$\nobreakdash-tor\-sion, $H$~and~$K$ are $\pi(\mathcal{C})^{\prime}$\nobreakdash-iso\-lat\-ed in~$G$.
\end{ecorollary}

\begin{ecorollary}\label{ec04}
Suppose that $\mathcal{C}$ is a~root class of~groups consisting only of~periodic groups and~closed under taking quotient groups, $\pi(\mathcal{C})$ contains all primes, $G$ is $\pi(\mathcal{C})$\nobreakdash-bound\-ed solvable, and~$H \ne G \ne K$. Suppose also that there exists a~number $m \geqslant 0$ such that $H_{m+1}$ and~$K_{m+1}$ are finitely generated or~$\pi^{\prime}$\nobreakdash-iso\-lat\-ed in~$K_{m}$ for~some finite set of~primes~$\pi$. Then $\mathfrak{G}$ is residually a~$\mathcal{C}$\nobreakdash-group if and~only if $H_{n} = K_{n}$ for~some $n \geqslant 1$.
\end{ecorollary}

Theorems~\ref{et01} and~\ref{et02} formulated above are in~fact corollaries of~Theorems~\ref{et03} and~\ref{et04}, which are given below and~use the~notion of~regularity of~a~group with~respect to~a~subgroup.{\parfillskip=0pt\par}

Suppose that $\mathcal{C}$ is a~class of~groups, $X$ is a~group, and~$Y$ is a~subgroup of~$X$. We say that $X$ is \emph{$\mathcal{C}$\nobreakdash-reg\-u\-lar} with~respect to~$Y$ if, for~any subgroup $M \in \mathcal{C}^{*}(Y)$, there exists a~subgroup $N \in \mathcal{C}^{*}(X)$ such that $M = N \cap Y$. The~notion of~regularity generalizes the~classical notion of~a~\emph{potent element}~\cite{Allenby1981}: if $\mathcal{F}$ is the~class of~all finite groups, then an~element $x \in X$ is potent if and~only if $X$ is $\mathcal{F}$\nobreakdash-reg\-u\-lar with~respect to~the~cyclic subgroup $\langle x \rangle$.

\begin{etheorem}\label{et03}
Suppose that $\mathcal{C}$ is a~root class of~groups consisting only of~periodic groups and~closed under taking quotient groups, $G$ is residually a~$\mathcal{C}$\nobreakdash-group, $H \ne G \ne K$,\linebreak and~$H_{n} = H_{n+1}$ for~some $n \geqslant 1$. Suppose also that, for~any $i \in \{0,\, 1,\, \ldots,\, n-1\}$, $K_{i}$ is $\mathcal{C}$\nobreakdash-reg\-u\-lar with~respect to~$H_{i+1}K_{i+1}$ and~$\mathcal{I}_{\pi(\mathcal{C})^{\prime}}(K_{i}, H_{i+1}K_{i+1})$ is $\mathcal{C}$\nobreakdash-sep\-a\-ra\-ble in~$K_{i}$. Then $\mathfrak{G}$ is residually a~$\mathcal{C}$\nobreakdash-group if and~only~if

\textup{1)}\hspace{1ex}$H_{n} = K_{n}$\textup{;}

\textup{2)}\hspace{1ex}the~subgroup $E = \operatorname{sgp}\{H_{n}, t\}$ is residually a~$\mathcal{C}$\nobreakdash-group\textup{;}

\textup{3)}\hspace{1ex}$H$~and~$K$ are $\pi(\mathcal{C})^{\prime}$\nobreakdash-iso\-lat\-ed in~$G$.
\end{etheorem}

\begin{etheorem}\label{et04}
Suppose that $\mathcal{C}$ is a~root class of~groups containing non-pe\-ri\-od\-ic groups and~closed under taking quotient groups, $G$ is residually a~$\mathcal{C}$\nobreakdash-group, and,~for~some $n \geqslant 0$, there exists a~subgroup $Q \in \mathcal{C}^{*}(K_{n})$ satisfying at~least one of~the~following conditions:

$(\alpha)$\hspace{1ex}$H_{n+1} \cap Q = 1 = K_{n+1} \cap Q$,

$(\beta)$\hspace{1ex}$Q \leqslant H_{n+1} \cap K_{n+1}$ and~$Q\varphi = Q$.

\noindent
Suppose also that, for~any $i \in \{0,\, 1,\, \ldots,\, n-1\}$, $K_{i}$ is $\mathcal{C}$\nobreakdash-reg\-u\-lar with~respect to~$H_{i+1}K_{i+1}$ and~$H_{i+1}K_{i+1}$ is $\mathcal{C}$\nobreakdash-sep\-a\-ra\-ble in~$K_{i}$. Then $\mathfrak{G}$ is residually a~$\mathcal{C}$\nobreakdash-group.
\end{etheorem}

The~above theorems and~corollaries generalize the~main results from~\cite{AndreadakisRaptisVarsos1988, Goltsov2015, SokolovTumanova2017} (as~for~\cite{AndreadakisRaptisVarsos1988}, in~the~part concerning non-as\-cend\-ing HNN-ex\-ten\-sions), as~well as~Theorem~1.1 from~\cite{RaptisVarsos1989} and~Theorem~3 from~\cite{SokolovTumanova2019}. However, the~results mentioned are easier to~formulate due to~the~additional restrictions imposed on~the~base group and~the~associated subgroups. The~proofs of~Theorems~\ref{et01}\nobreakdash--\ref{et04} and~Corollaries~\ref{ec01}\nobreakdash--\ref{ec04} are given in~Sections~\ref{esec02}\nobreakdash--\ref{esec06}.

\section{Generalized direct products of~groups}\label{esec02}

Let $\Gamma = (V, E)$ be a~non-empty connected undirected graph with~a~vertex set~$V$ and~an~edge set~$E$. It is assumed that $\Gamma$ is not necessarily finite, but has no multiple edges and~loops. Let us assign to~each vertex $v \in V$ a~group~$G_{v}$ and~to each edge $e = \{v, w\} \in E$ a~group~$H_{e}$ and~injective homomorphisms $\varphi_{e, v}\colon H_{e} \to G_{v}$, $\varphi_{e, w}\colon H_{e} \to G_{w}$. As~a~result, we get a~\emph{graph of~groups~$\mathcal{G}(\Gamma)$}. We call the~groups~$G_{v}$ ($v \in V$), $H_{e}$ ($e \in E$), the~subgroups~$H_{e}\varphi_{e, v}$, and~the~homomorphisms~$\varphi_{e, v}$ ($e \in E$, $v \in e$) the~\emph{vertex} and~\emph{edge groups}, the~\emph{edge subgroups}, and~the~\emph{edge homomorphisms}, respectively.

Consider the~group
\begin{multline*}
\mathrm{GDP}(\mathcal{G}(\Gamma)) = \big\langle G_{\lambda}\ (\lambda \in V);\ [G_{\mu}, G_{\eta}] = 1\ (\mu, \eta \in V,\ \mu \ne \eta),\\ H_{e}\varphi_{e, v} = H_{e}\varphi_{e, w}\ (e = \{v, w\} \in E) \big\rangle,
\end{multline*}
whose generators are the~generators of~$G_{\lambda}$ ($\lambda \in V$), and~whose defining relations are the~relations of~$G_{\lambda}$ ($\lambda \in V$) and~all possible relations of~the~form 
\[
[g_{\mu}, g_{\eta}] = 1\ (\mu, \eta \in V,\ \mu \ne \eta),\quad h_{e}\varphi_{e, v} = h_{e}\varphi_{e, w}\ (e = \{v, w\} \in E,\ h_{e} \in H_{e}),
\]
where $g_{\mu}$ and~$g_{\eta}$ are arbitrary words in~the~generators of~$G_{\mu}$ and~$G_{\eta}$, respectively, $h_{e}\varphi_{e, v}$ and~$h_{e}\varphi_{e, w}$ are some words in~the~generators of~$G_{v}$ and~$G_{w}$ that define (in~these groups) the~images of~$h_{e}$ under $\varphi_{e, v}$ and~$\varphi_{e, w}$. Following~\cite{SokolovTumanova2019}, we call the~group $\mathrm{GDP}(\mathcal{G}(\Gamma))$ \emph{the~generalized direct product associated with~the~graph of~groups~$\mathcal{G}(\Gamma)$}~if
\smallskip

$(i)\phantom{i}$\hspace{1ex}for~any $v \in V$, the~identity mapping of~the~generators of~$G_{v}$ to~$\mathrm{GDP}(\mathcal{G}(\Gamma))$ can be extended to~an~injective homomorphism, and~therefore all the~groups~$G_{v}$ ($v \in V$) can be considered subgroups of~$\mathrm{GDP}(\mathcal{G}(\Gamma))$;
\smallskip

$(ii)$\hspace{1ex}for~any $e\kern-1pt{} =\kern-1.5pt{} \{v, w\}\kern-1.5pt{} \in\kern-1.5pt{} E$, the~equalities $H_{e}\varphi_{e, v}\kern-1pt{} =\kern-1pt{} G_{v} \cap G_{w}\kern-1pt{} =\kern-1.5pt{} H_{e}\varphi_{e, w}$ hold in~$\mathrm{GDP}(\mathcal{G}(\Gamma))$.%
\smallskip

We say that the~generalized direct product associated with~$\mathcal{G}(\Gamma)$ \emph{exists} if $\mathrm{GDP}(\mathcal{G}(\Gamma))$ satisfies~$(i)$ and~$(ii)$. Some conditions for~the~existence of~generalized direct products are found in~\cite{SokolovTumanova2019}. In~particular, the~following proposition is proved.

\begin{eproposition}\label{ep201}
\textup{\cite[Theorem~1]{SokolovTumanova2019}}
If $\Gamma$ is a~tree and,~for~any $e \in E$ and~$v \in e$,~the~subgroup\kern-1pt{} $H_{e}\varphi_{e, v}$\kern-1pt{} lies\kern-.5pt{} in\kern-.5pt{}~the\kern-.5pt{}~center\kern-.5pt{} of~$G_{v}$\kern-.5pt{},\kern-1.5pt{} then\kern-.5pt{} the\kern-.5pt{}~generalized\kern-.5pt{} direct\kern-.5pt{} product\kern-.5pt{} associated\kern-.5pt{}~with\kern-.5pt{}~$\mathcal{G}(\Gamma)$ exists.
\end{eproposition}

The~main aim of~this section is to~find conditions for~the~existence of~certain generalized direct products associated with~simple cycles. 

Let $(X, Y, Z, \psi)$ be an~HNN-tuple, and~let $\Gamma$ be the~simple cycle of~length $n \geqslant 3$ with~the~vertex set $V = \mathbb{Z}_{n}$ and~the~edge set $E = \big\{\{i-1,\ i\} \mid i \in \mathbb{Z}_{n}\big\}$. Suppose also that, for~any $i \in \mathbb{Z}_{n}$, $X_{i}$ is an~isomorphic copy of~$X$ and~$\sigma_{i}\colon X \to X_{i}$ is an~isomorphism. We associate a~vertex $i \in \mathbb{Z}_{n}$ with~the~group~$X_{i}$, an~edge $e = \{i-1,\ i\} \in E$ with~the~group~$Y$ and~the~homomorphisms $\varphi_{e, i-1} = \sigma_{i-1}\vert_{Y}$, $\varphi_{e, i} = \psi\sigma_{i}\vert_{Z}$ (here and~below, all indices are considered modulo~$n$), and~denote the~resulting graph of~groups by~$\mathcal{G}_{n}(X, Y, Z, \psi)$.

Let us call a~number $n \geqslant 3$ \emph{admissible for~an~HNN-tuple $(X, Y, Z, \psi)$ with~a~reserve~$r$} ($0 \leqslant r \leqslant n-2$) if the~following conditions hold:
\smallskip

$(i)^{\prime}\phantom{i}$\hspace{1ex}for~any $i \in \mathbb{Z}_{n}$, the~identity mapping of~the~generators of~$X_{i}$ to~$\mathrm{GDP}(\mathcal{G}_{n}(X, Y, Z, \psi))$ can be extended to~an~injective homomorphism;
\smallskip

$(ii)^{\prime}$\hspace{1ex}for~any $s \in \{0,\, 1,\, \ldots,\, r\}$, $q \in \mathbb{Z}_{n}$, $x_{q} \in X_{q}$, $x_{q+1} \in X_{q+1}$, \ldots, $x_{q+s+1} \in X_{q+s+1}$, it follows from~the~equality $x_{q}x_{q+1}\ldots x_{q+s+1} = 1$ that $x_{q} \in Y\sigma_{q}$ and~$x_{q+s+1} \in Z\sigma_{q+s+1}$.
\smallskip

We note that if a~number~$n$ is admissible for~a~tuple $(X, Y, Z, \psi)$ with~a~reserve $r \in\nolinebreak \{0,\, \ldots,\, n-2\}$, then the~generalized direct product associated with~$\mathcal{G}_{n}(X, Y, Z, \psi)$ exists. Indeed, for~any $e = \{i-1,\ i\} \in E$, if $x \in X_{i-1} \cap X_{i}$, $x_{i-1} = x$, and~$x_{i} = x^{-1}$, then $x_{i-1}x_{i} = 1$ and,~by~the~admissibility,
$$
x_{i-1} \in Y\sigma_{i-1} = Y\varphi_{e, i-1},\quad
x_{i} \in Z\sigma_{i} = Y\varphi_{e, i}.
$$
Thus, $X_{i-1} \cap X_{i} \leqslant Y\varphi_{e, i-1} \cap Y\varphi_{e, i}$ and~therefore $Y\varphi_{e, i-1} = X_{i-1} \cap X_{i} = Y\varphi_{e, i}$.

\begin{eproposition}\label{ep202}
Suppose that $(G, H, K, \varphi)$ is an~HNN-tuple, $H$~and~$K$ lie in~the~center of~$G$, $L = H \cap K$, and~$M = L\varphi$. Then a~number $n \geqslant 3$ is admissible for~$(G, H, K, \varphi)$ with~a~reserve~$r \leqslant n-3$ if and~only if it is admissible for~$(K, L, M, \varphi)$ with~the~reserve~$r+1$.
\end{eproposition}

\begin{proof}
For each $i \in \mathbb{Z}_{n}$, let $G_{i}$ denote an~isomorphic copy of~$G$, and~let $\sigma_{i}\colon G \to G_{i}$ be an~isomorphism. We put
$$
H_{i} = H\sigma_{i},\quad K_{i} = K\sigma_{i},\quad L_{i} = L\sigma_{i},\quad M_{i} = M\sigma_{i},\quad \varphi_{i} = (\sigma_{i-1}\vert_{H_{i-1}})^{-1}\varphi\sigma_{i}\vert_{K}
$$
(in~Section~\ref{esec01}, the~symbols $H_{i}$ and~$K_{i}$ correspond to~other subgroups; the~notation just introduced is valid only throughout this proof). Then $\varphi_{i}$ is an~isomorphism of~$H_{i-1}$ onto~$K_{i}$, and~its restriction to~$L_{i-1}$ is an~isomorphism of~$L_{i-1}$ and~$M_{i}$. If $\mathrm{GDP}(\mathcal{G}_{n}(G, H, K, \varphi))$ satisfies $(i)^{\prime}$, then it can be considered containing~$H_{i}$ and~$K_{i}$ ($i \in \mathbb{Z}_{n}$), and~it follows from~the~relations $h\sigma_{i-1} = h\varphi\sigma_{i}$ ($i \in \mathbb{Z}_{n}$, $h \in H$) that the~equalities $h\varphi_{i} = h$ ($i \in \mathbb{Z}_{n}$, $h \in H_{i-1}$) hold in~$\mathrm{GDP}(\mathcal{G}_{n}(G, H, K, \varphi))$. Similarly, if $\mathrm{GDP}(\mathcal{G}_{n}(K, L, M, \varphi))$ satisfies $(i)^{\prime}$, then the~equalities $h\varphi_{i} = h$ ($i \in \mathbb{Z}_{n}$, $h \in L_{i-1}$) hold in~it.
\smallskip

\textit{Sufficiency}. We fix a~number $i \in \mathbb{Z}_{n}$ and~define a~mapping~$\theta_{i}$ of~the~subgroup $K_{i}H_{i} \leqslant G_{i}$ to~the~subgroup $K_{i}K_{i+1} \leqslant \mathrm{GDP}(\mathcal{G}_{n}(K, L, M, \varphi))$ as~follows: if $h \in H_{i}$ and~$k \in K_{i}$, then $(kh)\theta_{i} = k(h\varphi_{i+1})$. Let us show that this mapping is well defined and~is a~subgroup isomorphism extending the~identity mapping of~$K_{i}$.

{\parfillskip=0pt
If $h_{1}^{\vphantom{1}}, h_{2}^{\vphantom{1}} \in H_{i}^{\vphantom{1}}$, $k_{1}^{\vphantom{1}}, k_{2}^{\vphantom{1}} \in K_{i}^{\vphantom{1}}$, and~$k_{1}^{\vphantom{1}}h_{1}^{\vphantom{1}} = k_{2}^{\vphantom{1}}h_{2}^{\vphantom{1}}$, then $k_{2}^{-1}k_{1}^{\vphantom{1}} = h_{2}^{\vphantom{1}}h_{1}^{-1} \in H_{i}^{\vphantom{1}} \cap K_{i}^{\vphantom{1}} = L_{i}^{\vphantom{1}}$ and~therefore the~equality $(h_{2}^{\vphantom{1}}h_{1}^{-1})\varphi_{i+1}^{\vphantom{1}} = h_{2}^{\vphantom{1}}h_{1}^{-1}$ holds in~$\mathrm{GDP}(\mathcal{G}_{n}(K, L, M, \varphi))$. Hence,
\begin{multline*}
(k_{1}^{\vphantom{1}}h_{1}^{\vphantom{1}})\theta_{i}^{\vphantom{1}} = k_{1}^{\vphantom{1}}(h_{1}^{\vphantom{1}}\varphi_{i+1}^{\vphantom{1}}) = k_{2}^{\vphantom{1}}(k_{2}^{-1}k_{1}^{\vphantom{1}})((h_{1}^{\vphantom{1}}h_{2}^{-1})h_{2}^{\vphantom{1}})\varphi_{i+1}^{\vphantom{1}} \\
= k_{2}^{\vphantom{1}}(k_{2}^{-1}k_{1}^{\vphantom{1}}h_{1}^{\vphantom{1}}h_{2}^{-1})(h_{2}^{\vphantom{1}}\varphi_{i+1}^{\vphantom{1}}) = k_{2}^{\vphantom{1}}(h_{2}^{\vphantom{1}}\varphi_{i+1}^{\vphantom{1}}) = (k_{2}^{\vphantom{1}}h_{2}^{\vphantom{1}})\theta_{i}^{\vphantom{1}}.
\end{multline*}

Suppose that $h \in H_{i}$, $k \in K_{i}$, and~$1 = (kh)\theta_{i} = k(h\varphi_{i+1})$. Since $k \in K_{i}$, $h\varphi_{i+1} \in K_{i+1}$, and~$n$ is admissible for~$(K, L, M, \varphi)$, then $k \in L_{i}$ and~$h\varphi_{i+1} \in M_{i+1}$. Hence, $h \in L_{i}$ and~$1 = k(h\varphi_{i+1}) = kh$ because the~equality $h\varphi_{i+1} = h$ holds in~$\mathrm{GDP}(\mathcal{G}_{n}(K, L, M, \varphi))$.

Thus, $\theta_{i}$ is well defined and~injective, it is clear that it is homomorphic and~surjective.

}Let $\Delta$ be the~star graph with~the~vertex set $\{v_{i}\ (i \in \mathbb{Z}_{n}),\ w\}$. We associate its central vertex~$w$ with~the~group $\mathrm{GDP}(\mathcal{G}_{n}(K, L, M, \varphi))$, the~leaf~$v_{i}$ ($i \in \mathbb{Z}_{n}$) with~the~group~$G_{i}$, the~edge~$\{w, v_{i}\}$ ($i \in \mathbb{Z}_{n}$) with~the~subgroup~$K_{i}H_{i}$ of~$G_{i}$ and~the~homomorphisms, one of~which is the~identity mapping and~the~other coincides with~$\theta_{i}$. Let us denote the~resulting graph by~$\mathcal{G}(\Delta)$.

It follows from~the~definitions of~$\mathcal{G}(\Delta)$, $\theta_{i}$ ($i \in \mathbb{Z}_{n}$), and~$\varphi_{i}$ ($i \in \mathbb{Z}_{n}$) that, for~all $i \in \mathbb{Z}_{n}$ and~$h \in H$, the~equalities $h\sigma_{i-1} = (h\sigma_{i-1})\theta_{i-1} = (h\sigma_{i-1})\varphi_{i} = h\varphi \sigma_{i}$ hold in~$\mathrm{GDP}(\mathcal{G}(\Delta))$. Therefore, the~identity mapping of~the~generators of~$\mathrm{GDP}(\mathcal{G}_{n}(G, H, K, \varphi))$ to~$\mathrm{GDP}(\mathcal{G}(\Delta))$ defines a~homomorphism, which we denote by~$\lambda$.

Let, for~any $i \in \mathbb{Z}_{n}$, $\alpha_{i}\colon G_{i} \to \mathrm{GDP}(\mathcal{G}(\Delta))$ and~$\beta_{i}\colon G_{i} \to \mathrm{GDP}(\mathcal{G}_{n}(G, H, K, \varphi))$ be the~homomorphisms defined by~the~identity mappings of~the~generators of~$G_{i}$. Since the~diagram
\[
\begin{CD}
G_{i} @>{\beta_{i}}>>  \mathrm{GDP}(\mathcal{G}_{n}(G, H, K, \varphi))\\
@|                     @VV{\lambda}V\\
G_{i} @>{\alpha_{i}}>> \mathrm{GDP}(\mathcal{G}(\Delta))
\end{CD}
\]
is commutative and~$\alpha_{i}$ is injective by~Proposition~\ref{ep201}, then $\beta_{i}$ is also injective and~therefore $\mathrm{GDP}(\mathcal{G}_{n}(G, H, K, \varphi))$ satisfies~$(i)^{\prime}$. Let us verify that this group also satisfies~$(ii)^{\prime}$.

Suppose that numbers $q \in \mathbb{Z}_{n}$, $s \in \{0,\, 1,\, \ldots,\, r\}$ and~elements
$$
g_{q} \in G_{q},\ g_{q+1} \in G_{q+1},\ \ldots,\ g_{q+s+1} \in G_{q+s+1}
$$
are such that $g_{q}g_{q+1}\ldots g_{q+s+1} = 1$ in~$\mathrm{GDP}(\mathcal{G}_{n}(G, H, K, \varphi))$. Then this equality also holds in~$\mathrm{GDP}(\mathcal{G}(\Delta))$ by~the~definition of~$\lambda$.

Let $j \in \{0,\, 1,\, \ldots,\, s+1\}$, and~let $\mathcal{G}_{q+j}$ be the~graph of~groups obtained from~$\mathcal{G}(\Delta)$ by~deleting the~vertex~$v_{q+j}$ and~the~edge~$\{v_{q+j}, w\}$. Then $\mathrm{GDP}(\mathcal{G}(\Delta))$ is the~generalized direct product~$\mathfrak{P}_{j}$ of~the~groups $\mathrm{GDP}(\mathcal{G}_{q+j})$ and~$G_{q+j}$ with~the~amalgamated subgroups~$K_{q+j}K_{q+j+1}$ and~$K_{q+j}H_{q+j}$. It follows from~the~relation $r \leqslant n-3$ that $s+2 < n$ and~therefore
\begin{align*}
g_{q+j-1}^{-1}g_{q+j-2}^{-1}\ldots g_{q}^{-1}g_{q+s+1}^{-1}\ldots g_{q+j+2}^{-1}g_{q+j+1}^{-1} &\in \mathrm{GDP}(\mathcal{G}_{q+j}^{\vphantom{1}}),\\
g_{q+j}^{\phantom{-1}} = g_{q+j-1}^{-1}g_{q+j-2}^{-1}\ldots g_{q}^{-1}g_{q+s+1}^{-1}\ldots g_{q+j+2}^{-1}g_{q+j+1}^{-1} &\in \mathrm{GDP}(\mathcal{G}_{q+j}^{\vphantom{1}}) \cap G_{q+j}^{\vphantom{1}}.
\end{align*}
Since $\mathfrak{P}_{j}$ satisfies~$(ii)$ by~Proposition~\ref{ep201}, we have $g_{q+j} \in K_{q+j}H_{q+j}$.

For each $j \in \{0,\, 1,\, \ldots,\, s+1\}$, let us write the~element~$g_{q+j}$ in~the~form $g_{q+j} =\nolinebreak k_{q+j}h_{q+j}$, where $h_{q+j} \in H_{q+j}$, $k_{q+j} \in K_{q+j}$. Since $\mathrm{GDP}(\mathcal{G}_{n}(G, H, K, \varphi))$ satisfies $(i)^{\prime}$, then the~relations $h_{q+j} = h_{q+j}\varphi_{q+j+1} \in K_{q+j+1}$ ($0 \leqslant j \leqslant s+1$) hold in~it and~therefore 
\[
k_{q} \in K_{q},\ h_{q}k_{q+1} \in K_{q+1},\ \ldots,\ h_{q+s}k_{q+s+1} \in K_{q+s+1},\ h_{q+s+1} \in K_{q+s+2}.
\]

As~noted above, the~equality $g_{q}g_{q+1}\ldots g_{q+s+1} = 1$ holds in~$\mathrm{GDP}(\mathcal{G}(\Delta))$. We rewrite it in~the~form
\[
k_{q}(h_{q}k_{q+1})\ldots (h_{q+s}k_{q+s+1})h_{q+s+1} = 1.
\]

By~Proposition~\ref{ep201}, $\mathrm{GDP}(\mathcal{G}_{n}(K, L, M, \varphi))$ is embedded into~$\mathrm{GDP}(\mathcal{G}(\Delta))$ by~the~identity mapping of~the~generators. Therefore, the~last relation holds in~$\mathrm{GDP}(\mathcal{G}_{n}(K, L, M, \varphi))$, and~since $n$ is admissible for~$(K, L, M, \varphi)$ with~the~reserve~$r+1$, then $h_{q+s+1}^{\vphantom{1}} \in M_{q+s+2}^{\vphantom{1}}$ and~$k_{q}^{\vphantom{1}} \in L_{q}^{\vphantom{1}}$. Hence, $h_{q+s+1}^{\vphantom{1}}\varphi_{q+s+2}^{-1} \in L_{q+s+1}^{\vphantom{1}} \leqslant K_{q+s+1}^{\vphantom{1}}$~and
\[
g_{q}^{\vphantom{1}} = k_{q}^{\vphantom{1}}h_{q}^{\vphantom{1}} \in H_{q}^{\vphantom{1}},\quad 
g_{q+s+1}^{\vphantom{1}} = k_{q+s+1}^{\vphantom{1}}h_{q+s+1}^{\vphantom{1}} = k_{q+s+1}^{\vphantom{1}}(h_{q+s+1}^{\vphantom{1}}\varphi_{q+s+2}^{-1}) \in K_{q+s+1}^{\vphantom{1}},
\]
as~required.
\smallskip

\textit{Necessity.} It is easy to~see that the~mapping of~words acting on~the~generators of the~groups~$K_{i}$ ($i \in \mathbb{Z}_{n}$) as~the~natural embeddings $\iota_{i}\colon K_{i} \to G_{i}$ ($i \in \mathbb{Z}_{n}$) takes all the~defining relations of~$\mathrm{GDP}(\mathcal{G}_{n}(K, L, M, \varphi))$ to~the~equalities valid in~$\mathrm{GDP}(\mathcal{G}_{n}(G, H, K, \varphi))$ and~therefore defines a~homomorphism~$\mu$ from~the~first group to~the~second.

Let $\beta_{i}\colon G_{i} \to \mathrm{GDP}(\mathcal{G}_{n}(G, H, K, \varphi))$ and~$\gamma_{i}\colon K_{i} \to \mathrm{GDP}(\mathcal{G}_{n}(K, L, M, \varphi))$ be the~homomorphisms defined by~the~identity mappings of~the~generators of~$G_{i}$ and~$K_{i}$. Then, for~any $i \in \mathbb{Z}_{n}$, the~diagram
\[
\begin{CD}
K_{i} @>{\gamma_{i}}>> \mathrm{GDP}(\mathcal{G}_{n}(K, L, M, \varphi))\\
@V{\iota_{i}}VV        @VV{\mu}V\\
G_{i} @>{\beta_{i}}>>  \mathrm{GDP}(\mathcal{G}_{n}(G, H, K, \varphi))
\end{CD}
\]
is commutative, and~since the~homomorphisms~$\beta_{i}$ ($i \in \mathbb{Z}_{n}$) are injective, the~homomorphisms~$\gamma_{i}$ ($i \in \mathbb{Z}_{n}$) are also injective. Thus, the~group $\mathrm{GDP}(\mathcal{G}_{n}(K, L, M, \varphi))$ satisfies~$(i)^{\prime}$.

Suppose that numbers $q \in \mathbb{Z}_{n}$, $s \in \{0,\, 1,\, \ldots,\, r+1\}$ and~elements
$$
k_{q} \in K_{q},\ k_{q+1} \in K_{q+1},\ \ldots,\ k_{q+s+1} \in K_{q+s+1}
$$
are such that $k_{q}k_{q+1}\ldots k_{q+s+1} = 1$. It follows from~the~definition of~$\mu$ that the~last equality also holds in~$\mathrm{GDP}(\mathcal{G}_{n}(G, H, K, \varphi))$. Since $n$ is admissible for~$(G, H, K, \varphi)$ with~the~reserve~$r$~and
\[
k_{q}^{\vphantom{1}} \in G_{q}^{\vphantom{1}},\ \ldots,\ k_{q+s-1}^{\vphantom{1}} \in G_{q+s-1}^{\vphantom{1}},\ k_{q+s}^{\vphantom{1}}k_{q+s+1}^{\vphantom{1}} = k_{q+s}^{\vphantom{1}}(k_{q+s+1}^{\vphantom{1}}\varphi_{q+s+1}^{-1}) \in G_{q+s}^{\vphantom{1}},
\]
then $k_{q}^{\vphantom{1}}\kern-2pt{} \in\kern-2pt{} H_{q}^{\vphantom{1}}$ and~$k_{q+s}^{\vphantom{1}}(k_{q+s+1}^{\vphantom{1}}\varphi_{q+s+1}^{-1})\kern-2pt{} \in\kern-2pt{} K_{q+s}^{\vphantom{1}}$.\kern-3pt{} The~first relation means that $k_{q}^{\vphantom{1}}\kern-2pt{} \in\kern-2pt{} H_{q}^{\vphantom{1}} \cap\nolinebreak K_{q}^{\vphantom{1}}\kern-1.5pt{} =\nolinebreak\kern-1.5pt{} L_{q}^{\vphantom{1}}$, the~second one implies that $k_{q+s+1}^{\vphantom{1}}\varphi_{q+s+1}^{-1} \in K_{q+s}^{\vphantom{1}}$ and~therefore
$$
k_{q+s+1}^{\vphantom{1}}\varphi_{q+s+1}^{-1} \in H_{q+s}^{\vphantom{1}} \cap K_{q+s}^{\vphantom{1}} = L_{q+s}^{\vphantom{1}}.
$$
Hence, $k_{q+s+1}^{\vphantom{1}} \in M_{q+s+1}^{\vphantom{1}}$, as~required.
\end{proof}

\begin{eproposition}\label{ep203}
Suppose that $(G, H, K, \varphi)$ is an~HNN-tuple, $G$ is an~abelian group, and~$H = G = K$. If the~order~$q$ of~the~automorphism~$\varphi$ is finite and~divides a~number~$n \geqslant\nolinebreak 3$, then $n$ is admissible for~$(G, H, K, \varphi)$ with~any reserve~$r \in \{0,\, 1,\, \ldots,\, n-2\}$.\parfillskip=0pt
\end{eproposition}

\begin{proof}
For each $i \in \mathbb{Z}_{n}$, let again $G_{i}$ denote an~isomorphic copy of~$G$, and~let $\sigma_{i}\colon G \to G_{i}$ be an~isomorphism. Consider the~star graph~$\Delta$ with~the~vertex set $\{v_{i}\ (i \in \mathbb{Z}_{n}),\ w\}$ and~associate its central vertex~$w$ with~the~group~$G$, the~leaf~$v_{i}$ ($i \in \mathbb{Z}_{n}$) with~the~group~$G_{i}$, and~the~edge~$\{w, v_{i}\}$ ($i \in \mathbb{Z}_{n}$) with~the~group~$G$ and~the~homomorphisms, one of~which is the~identity mapping and~the~other coincides with~the~isomorphism~$\varphi^{i}\sigma_{i}$. Since $q \mid n$, it follows from~the~relation $x \equiv y \pmod n$ that $\varphi^{x} = \varphi^{y}$. Therefore, the~notation~$\varphi^{i}$ is correct.

We denote the~constructed graph of~groups by~$\mathcal{G}(\Delta)$. It follows from~its definition that, for~any $i \in \mathbb{Z}_{n}$ and~$g^{\prime} \in G$, the~equality $g^{\prime} = g^{\prime} \varphi^{i}\sigma_{i}$ holds in~$\mathrm{GDP}(\mathcal{G}(\Delta))$ and~therefore $g\sigma_{i-1} = g\varphi^{-(i-1)} = g\varphi \sigma_{i}$ for~all $i \in \mathbb{Z}_{n}$, $g \in G$. Hence, the~identity mapping of~the~generators of~$\mathrm{GDP}(\mathcal{G}_{n}(G, H, K, \varphi))$ to~$\mathrm{GDP}(\mathcal{G}(\Delta))$ can be extended to~a~homomorphism
$$
\lambda\colon \mathrm{GDP}(\mathcal{G}_{n}(G, H, K, \varphi)) \to \mathrm{GDP}(\mathcal{G}(\Delta)).
$$

Let $\alpha_{i}\colon G_{i} \to \mathrm{GDP}(\mathcal{G}(\Delta))$ and~$\beta_{i}\colon G_{i} \to \mathrm{GDP}(\mathcal{G}_{n}(G, H, K, \varphi))$ be the~homomorphisms defined by~the~identity mappings of~the~generators of~$G_{i}$ ($i \in \mathbb{Z}_{n}$). Then, for~any $i \in \mathbb{Z}_{n}$, the~diagram
\[
\begin{CD}
G_{i} @>{\beta_{i}}>>  \mathrm{GDP}(\mathcal{G}_{n}(G, H, K, \varphi))\\
@|                     @VV{\lambda}V\\
G_{i} @>{\alpha_{i}}>> \mathrm{GDP}(\mathcal{G}(\Delta)) 
\end{CD}
\]
is commutative. Since the~homomorphisms~$\alpha_{i}$ ($i \in \mathbb{Z}_{n}$) are injective by~Proposition~\ref{ep201}, then the~homomorphisms~$\beta_{i}$ ($i \in \mathbb{Z}_{n}$) are also injective. Therefore, $\mathrm{GDP}(\mathcal{G}_{n}(G, H, K, \varphi))$ satisfies $(i)^{\prime}$. It follows from~the~equalities $H = G = K$ that $(ii)^{\prime}$ is satisfied trivially.
\end{proof}

\begin{eproposition}\label{ep204}
\mbox{}\kern-2pt{}Suppose that a~number $n\kern-2pt{} \geqslant\kern-2pt{} 3$ is admissible for~an~HNN-tuple $(G\kern-.5pt{},\kern-.5pt{} H\kern-1pt{},\kern-.5pt{} K\kern-1pt{},\kern-.5pt{} \varphi)$ with~a~reserve~$r \in \{0,\, 1,\, \ldots,\, n-2\}$ and~$C_{n}$ is a~cyclic group of~order~$n$ with~a~generator~$c$. Suppose also that $\mathcal{C}$ is a~class of~groups closed under taking quotient groups and~extensions. If $G \in \mathcal{C}$ and~$C_{n} \in \mathcal{C}$, then there exists a~homomorphism of~the~group $\mathfrak{G} = \mathrm{HNN}(G, H, K, \varphi)$ onto~a~group from~$\mathcal{C}$ acting injectively on~$G$.
\end{eproposition}

\begin{proof}
Let $G_{i}$ and~$\sigma_{i}$ ($i \in \mathbb{Z}_{n}$) be defined as~above. Then the~generalized direct product $P = \mathrm{GDP}(\mathcal{G}_{n}(G, H, K, \varphi))$ has the~representation
\[
P = \big\langle G_{i}\ (i \in \mathbb{Z}_{n});\ [G_{i}, G_{j}] = 1\ (i, j \in \mathbb{Z}_{n},\ i \ne j),\ H\sigma_{i-1} = K\sigma_{i}\ (i \in \mathbb{Z}_{n}) \big\rangle,
\]
which shows that the~map extending the~isomorphisms $\sigma_{i}^{-1}\sigma_{i-1}^{\vphantom{1}}\colon G_{i}^{\vphantom{1}} \to G_{i-1}^{\vphantom{1}}$ defines an~automorphism~$\alpha$ of~this group. Obviously, the~order of~this automorphism divides~$n$, and~therefore there exists a~split extension~$Q$ of~$P$ by~$C_{n}$ such that $\hat c\vert_{P} = \alpha$ (here and~below, $\hat x$ denotes the~inner automorphism defined by~an~element~$x$). Since $G \in \mathcal{C}$, then $P$ is a~quotient group of~the~direct product of~$n$ $\mathcal{C}$\nobreakdash-groups. Hence, $P \in \mathcal{C}$ and~$Q \in \mathcal{C}$ because $\mathcal{C}$ is closed under taking quotient groups and~extensions.

It is easy to~see that the~map $\rho\colon \mathfrak{G} \to Q$ extending the~homomorphism $\sigma_{0}\colon G \to G_{0}$ and~mapping~$t$ to~$c$ takes all the~defining relations of~$\mathfrak{G}$ to~the~equalities valid in~$Q$ and~therefore is a~homomorphism. Since $c^{i}G_{0}c^{-i} = G_{i}$ ($i \in \mathbb{Z}_{n}$), this homomorphism is surjective. It remains to~note that, because $n$ is admissible, $\rho$ is injective on~$G$ and~thus is the~required mapping.
\end{proof}

\section{Compatible subgroups}\label{esec03}

Let $(G, H, K, \varphi)$ be an~HNN-tuple, and~let $\mathfrak{G} = \mathrm{HNN}(G, H, K, \varphi)$. Recall that a~subgroup $N \leqslant G$ is said to~be \emph{$(H, K, \varphi)$-com\-pat\-i\-ble} if ($N \cap H)\varphi = N \cap K$.

It is easy to~verify that if a~subgroup~$N$ is normal in~$G$ and~is $(H, K, \varphi)$-com\-pat\-i\-ble, then the~mapping $\varphi_{N}\colon HN/N \to KN/N$ taking a~coset~$hN$ ($h \in H$) to~$(h\varphi)N$ is well defined and~is a~subgroup isomorphism. Therefore, the~sequence
$$
(G/N, HN/N, KN/N, \varphi_{N})
$$
turns~out to~be an~HNN-tuple. It is also easy to~see that the~map
\[
\rho_{N}\colon \mathfrak{G} \to \mathrm{HNN}(G/N, HN/N, KN/N, \varphi_{N})
\]
extending the~natural homomorphism $G \to G/N$ and~taking~$t$ to~$t$ is a~surjective homomorphism, and~its kernel coincides with~the~normal closure of~$N$ in~$\mathfrak{G}$.

For~every class of~groups~$\mathcal{C}$, consider three families of~subgroups. Namely, suppose that%
\smallskip

--\hspace{1ex}$\mathcal{C}^{*}(G, H, K, \varphi)$ is the~family of~all $(H, K, \varphi)$-com\-pat\-i\-ble subgroups from~$\mathcal{C}^{*}(G)$;
\smallskip

--\hspace{1ex}$\mathcal{C}_{r}^{*}(G, H, K, \varphi)$ ($r \geqslant 0$) is the~subset of~$\mathcal{C}^{*}(G, H, K, \varphi)$ defined as~follows: a~subgroup $N \in \mathcal{C}^{*}(G, H, K, \varphi)$ belongs to~$\mathcal{C}_{r}^{*}(G, H, K, \varphi)$ if and~only if there exists a~number $n \geqslant \max\{3,\ r+2\}$ such that $\mathcal{C}$ contains a~cyclic group of~order~$n$ and~$n$ is admissible for~$(G/N, HN/N, KN/N, \varphi_{N})$ with~the~reserve~$r$;
\smallskip

--\hspace{1ex}$\mathcal{C}_{\cap}^{*}(G, H, K, \varphi) = \{U \cap G \mid U \in \mathcal{C}^{*}(\mathfrak{G})\}$.

\begin{eproposition}\label{ep301}
If $\mathcal{C}$ is a~class of~groups closed under taking subgroups and~direct~products of~a~finite number of~factors, and~$X$ is a~group, then the~following statements~hold.{\parfillskip=0pt\par}

\textup{1.}\hspace{1ex}The intersection of~a~finite number of~subgroups from~the~family $\mathcal{C}^{*}(X)$ again belongs to~this family~\textup{\cite[Proposition~2]{SokolovTumanova2020IVM};}

\textup{2.}\hspace{1ex}If $X$ is residually a~$\mathcal{C}$\nobreakdash-group and~$S$ is a~finite subset of~non-triv\-i\-al elements of~$X$, then there exists a~subgroup $Y \in \mathcal{C}^{*}(X)$ such that $Y \cap S = \varnothing$~\textup{\cite[Proposition~2]{Tumanova2020}}.
\end{eproposition}

\begin{eproposition}\label{ep302}
If $\mathcal{C}$ is a~class of~groups closed under taking subgroups and~direct products of~a~finite number of~factors, and~$(G, H, K, \varphi)$ is an~HNN-tuple, then the~families $\mathcal{C}^{*}(G, H, K, \varphi)$ and~$\mathcal{C}_{\cap}^{*}(G, H, K, \varphi)$ are closed under taking the~intersections of~a~finite number of~subgroups.
\end{eproposition}

\begin{proof}
An~obvious induction allows us to~consider the~intersection of~only two subgroups.

{\parfillskip=0pt
If $N_{1}, N_{2} \in \mathcal{C}^{*}(G, H, K, \varphi)$ and~$N = N_{1} \cap N_{2}$, then $N \in \mathcal{C}^{*}(G)$ by~Proposition~\ref{ep301}~and
\begin{multline*}
(N \cap H)\varphi = ((N_{1} \cap H) \cap (N_{2} \cap H))\varphi = (N_{1} \cap H)\varphi \cap (N_{2} \cap H)\varphi \\
= (N_{1} \cap K) \cap (N_{2} \cap K) = N \cap K
\end{multline*}
}because $\varphi$ is injective. Therefore, $N \in \mathcal{C}^{*}(G, H, K, \varphi)$.

Suppose that $N_{1}, N_{2} \in \mathcal{C}_{\cap}^{*}(G, H, K, \varphi)$, $\mathfrak{G} = \mathrm{HNN}(G, H, K, \varphi)$, and~$U_{1}, U_{2} \in \mathcal{C}^{*}(\mathfrak{G})$ are subgroups such that $N_{1} = U_{1} \cap G$, $N_{2} = U_{2} \cap G$. If $N = N_{1} \cap N_{2}$ and~$U = U_{1} \cap U_{2}$, then $N = U \cap G$ and,~again by~Proposition~\ref{ep301}, $U \in \mathcal{C}^{*}(\mathfrak{G})$. Thus, $N \in \mathcal{C}_{\cap}^{*}(G, H, K, \varphi)$.
\end{proof}

\begin{eproposition}\label{ep303}
\textup{\cite[Theorem~3]{SokolovTumanova2019}}
Suppose that $\mathcal{C}$ is a~root class of~groups containing non-pe\-ri\-od\-ic groups and~closed under taking quotient groups, $(G, H, K, \varphi)$ is an~HNN-tuple, and~$\mathfrak{G} = \mathrm{HNN}(G, H, K, \varphi)$. Suppose also that $G$ is residually a~$\mathcal{C}$\nobreakdash-group, $H$~and~$K$ lie in~the~center of~$G$, and~there exists a~homomorphism~$\rho$ of~$G$ onto~a~group from~$\mathcal{C}$ acting injectively on~$H$ and~$K$. Then $\rho$ can be extended to~a~homomorphism of~$\mathfrak{G}$ onto~a~group from~$\mathcal{C}$ and~$\mathfrak{G}$ is residually a~$\mathcal{C}$\nobreakdash-group.
\end{eproposition}\pagebreak

\begin{eproposition}\label{ep304}
Let $\mathcal{C}$ be a~class of~groups, and~let $(G, H, K, \varphi)$ be an~HNN-tuple.

\textup{1.}\hspace{1ex}If $\mathcal{C}$ is closed under taking subgroups, then $\mathcal{C}_{\cap}^{*}(G, H, K, \varphi) \subseteq \mathcal{C}^{*}(G, H, K, \varphi)$.

\textup{2.}\hspace{1ex}If $\mathcal{C}$ is a~root class of~groups containing non-pe\-ri\-od\-ic groups and~closed under taking quotient groups, $H$~and~$K$ lie in~the~center of~$G$, then $\mathcal{C}^{*}(G, H, K, \varphi) \subseteq \mathcal{C}_{\cap}^{*}(G, H, K, \varphi)$.{\parfillskip=0pt\par}

\textup{3.}\hspace{1ex}If $\mathcal{C}$ is closed under taking quotient groups and~extensions, then $\mathcal{C}_{r}^{*}(G, H, K, \varphi) \subseteq \mathcal{C}_{\cap}^{*}(G, H, K, \varphi)$ for~any $r \geqslant 0$.
\end{eproposition}

\begin{proof}
1.\hspace{1ex}Suppose that $\mathfrak{G} = \mathrm{HNN}(G, H, K, \varphi)$, $N \in \mathcal{C}_{\cap}^{*}(G, H, K, \varphi)$, and~$U \in \mathcal{C}^{*}(\mathfrak{G})$ is a~subgroup such that $N = U \cap G$. Then
$$
G/N = G/(U \cap G) \cong GU/U \leqslant \mathfrak{G}/U
$$
{\parfillskip=0pt
and~$G/N \in \mathcal{C}$ because $\mathcal{C}$ is closed under taking subgroups. Since $U$ is normal in~$\mathfrak{G}$, then
\[
(U \cap H)\varphi = t^{-1}(U \cap H)t \leqslant t^{-1}Ut \cap t^{-1}Ht = U \cap H\varphi = U \cap K
\]}
and~similarly $(U \cap K)\varphi^{-1} \leqslant U \cap H$. Hence,
\[
(N \cap H)\varphi = (U \cap G \cap H)\varphi = (U \cap H)\varphi = U \cap K = U \cap G \cap K = N \cap K
\]
and~therefore $N \in \mathcal{C}^{*}(G, H, K, \varphi)$.
\smallskip

2.\hspace{1ex}Let $N \in \mathcal{C}^{*}(G, H, K, \varphi)$. Then the~HNN-tuple $(G/N, HN/N, KN/N, \varphi_{N})$, the~group $\mathfrak{G}_{N} = \mathrm{HNN}(G/N, HN/N, KN/N, \varphi_{N})$, and~the~homomorphism $\rho_{N}\colon \mathfrak{G} \to \mathfrak{G}_{N}$ are defined. It follows from~the~definition of~the~family $\mathcal{C}^{*}(G, H, K, \varphi)$ that $G/N \in \mathcal{C}$. Hence, by~Proposition~\ref{ep303}, there exists a~homomorphism~$\sigma$ of~$\mathfrak{G}_{N}$ onto~a~group from~$\mathcal{C}$ extending the~identity mapping of~$G/N$. Let $U_{N} = \ker\sigma$, and~let $U$ be the~preimage of~$U_{N}$ under~$\rho_{N}$. Then $U_{N} \in \mathcal{C}^{*}(\mathfrak{G}_{N})$, $U_{N} \cap G/N = 1$, and~$U \in \mathcal{C}^{*}(\mathfrak{G})$. Since $\rho_{N}$ extends the~natural homomorphism $G \to G/N$, then $\ker\rho_{N} \cap G = N$ and~$U \cap G \leqslant \ker\rho_{N}$. Therefore, $U \cap G = N$ and~$N \in \mathcal{C}_{\cap}^{*}(G, H, K, \varphi)$.
\smallskip

3.\hspace{1ex}If $N \in \mathcal{C}_{r}^{*}(G, H, K, \varphi)$, then, by~Proposition~\ref{ep204}, there exists a~homomorphism of the~group $\mathrm{HNN}(G/N, HN/N, KN/N, \varphi_{N})$ onto~a~group from~$\mathcal{C}$ acting injectively on~$G/N$. Hence, we can use exactly the~same argument as~in~the~proof of~Statement~2.
\end{proof}

\begin{eproposition}\label{ep305}
Let $\mathcal{C}$ be a~root class of~groups consisting only of~periodic groups and~closed under taking quotient groups. If $(G, H, K, \varphi)$ is an~HNN-tuple, $G$ is an~abelian group, $H = G = K$, and~$N$ is a~subgroup of~$G$, then the~following statements are equivalent.

\textup{1.}\hspace{1ex}$N \in \mathcal{C}_{\cap}^{*}(G, H, K, \varphi)$.

\textup{2.}\hspace{1ex}$N \in \mathcal{C}_{r}^{*}(G, H, K, \varphi)$ for~any $r \geqslant 0$.

\textup{3.}\hspace{1ex}$N\varphi = N$, $G/N \in \mathcal{C}$, and~the~order of~the~automorphism~$\varphi_{N}$ of~$G/N$ induced by~the~automorphism~$\varphi$ is finite and~is a~$\pi(\mathcal{C})$\nobreakdash-num\-ber.
\end{eproposition}

\begin{proof}
\mbox{$1 \Rightarrow 3$}.\hspace{1ex}By~Proposition~\ref{ep304}, $N \in \mathcal{C}^{*}(G, H, K, \varphi)$. Hence, $G/N \in \mathcal{C}$, and~since $H = G = K$, then
$$
N\varphi = (N \cap H)\varphi = N \cap K = N.
$$

Let $\mathfrak{G} = \mathrm{HNN}(G, H, K, \varphi)$. Because $N \in \mathcal{C}_{\cap}^{*}(G, H, K, \varphi)$, there exists a~subgroup $U \in \mathcal{C}^{*}(\mathfrak{G})$ such that $N = U \cap G$. Since $\mathfrak{G}/U \in \mathcal{C}$ and~$\mathcal{C}$ consists only of~periodic groups, then the~order~$n$ of~$t$ modulo~$U$ is finite and~is a~$\pi(\mathcal{C})$\nobreakdash-num\-ber.

It follows from~the~equality $N = U \cap G$ that the~mapping $\gamma\colon G/N \to \mathfrak{G}/U$ taking a~coset~$gN$ to~$gU$ is well defined and~is an~injective homomorphism. For~any $g \in G$,~we~have
\[
(gN)\varphi_{N}^{n} = (g\varphi^{n})N = (gN)\gamma (\widehat{t^{n}U})\gamma^{-1} = gN.
\]

Hence, the~order of~$\varphi_{N}$ divides~$n$ and~therefore is a~$\pi(\mathcal{C})$\nobreakdash-num\-ber.
\pagebreak

\mbox{$3 \Rightarrow 2$}.\hspace{1ex}We fix a~number $r \geqslant 0$ and~choose a~$\pi(\mathcal{C})$\nobreakdash-num\-ber $n > r+3$ to~be a~multiple of~the~order of~$\varphi_{N}$ (it is possible to~find a~$\pi(\mathcal{C})$\nobreakdash-num\-ber greater than $r+3$ because $\mathcal{C}$ contains non-triv\-i\-al groups and~therefore $\pi(\mathcal{C}) \ne \varnothing)$. Since $H = G = K$, then
$$
(N \cap H)\varphi = N\varphi = N = N \cap K
$$
and~hence $N \in \mathcal{C}^{*}(G, H, K, \varphi)$. Therefore, the~tuple $(G/N, HN/N, KN/N, \varphi_{N})$ is defined and,~by~Proposition~\ref{ep203}, $n$ is admissible for~this tuple with~the~reserve~$r$. It remains to~note that $\mathcal{C}$ contains a~cyclic group of~order~$n$: this fact follows from~the~definition of~$\pi(\mathcal{C})$ and~the~assumption that $\mathcal{C}$ is closed under taking subgroups and~extensions.
\smallskip

\mbox{$2 \Rightarrow 1$}.\hspace{1ex}This implication is a~consequence of~Proposition~\ref{ep304}.
\end{proof}

\section{Descent and~ascent of~compatible subgroups}\label{esec04}

The~proofs of~the~propositions in~this section follow the~ideas of~\cite{Moldavanskii2002, Moldavanskii2003} and,~in some places, repeat the~arguments given in~the~above papers almost word for~word.

\begin{eproposition}\label{ep401}
Suppose that $\mathcal{C}$ is a~class of~groups closed under taking subgroups, $(G, H, K, \varphi)$ is an~HNN-tuple, $L = H \cap K$, $M = L\varphi$, and~$X$ is a~subgroup of~$G$ containing~$L$ and~$M$. Suppose also that $N$ is a~subgroup of~$G$ and~$R = N \cap X$. Then the~following statements hold.

\textup{1.}\hspace{1ex}If $N \in \mathcal{C}^{*}(G, H, K, \varphi)$, then $R \in \mathcal{C}^{*}(X, L, M, \varphi)$.

\textup{2.}\hspace{1ex}If $N \in \mathcal{C}_{\cap}^{*}(G, H, K, \varphi)$, then $R \in \mathcal{C}_{\cap}^{*}(X, L, M, \varphi)$.
\end{eproposition}

\begin{proof}
1.\hspace{1ex}We have $X/R \in \mathcal{C}$ because
$$
X/R = X/(N \cap X) \cong XN/N \leqslant G/N \in \mathcal{C}
$$
and~$\mathcal{C}$ is closed under taking subgroups. Since $\varphi$ is injective and~$N$ is $(H, K, \varphi)$-com\-pat\-i\-ble, then
$$
((N \cap H) \cap L)\varphi = (N \cap K) \cap M.
$$
It follows from~this equality and~the~relation $L \cup M \leqslant X$ that
\begin{multline*}
(R \cap L)\varphi = ((N \cap X) \cap L)\varphi = (N \cap L)\varphi = ((N \cap H) \cap L)\varphi \\
= (N \cap K) \cap M = N \cap M = (N \cap X) \cap M = R \cap M.
\end{multline*}

Thus, $R \in \mathcal{C}^{*}(X, L, M, \varphi)$.
\smallskip

2.\hspace{1ex}If $\mathfrak{G} = \mathrm{HNN}(G, H, K, \varphi)$ and~$\mathfrak{X} = \mathrm{HNN}(X, L, M, \varphi)$, then the~map $\lambda\colon \mathfrak{X} \to \mathfrak{G}$ taking the~generators of~$\mathfrak{X}$ to~the~corresponding elements of~$\mathfrak{G}$ defines a~homomorphism, which acts identically on~$X$. Let $U \in \mathcal{C}^{*}(\mathfrak{G})$ be a~subgroup such that $N = U \cap G$, and~suppose that $\tilde{\mathfrak{X}} = \mathfrak{X}\lambda$, $\tilde V = U \cap \tilde{\mathfrak{X}}$, $V$ is the~full preimage of~$\tilde V$ under~$\lambda$. Then $\tilde{\mathfrak{X}}/\tilde V$ and~$\mathfrak{X}/V$ are isomorphic groups, which belong to~$\mathcal{C}$ because
$$
\tilde{\mathfrak{X}}/\tilde V = \tilde{\mathfrak{X}}/(U \cap \tilde{\mathfrak{X}}) \cong \tilde{\mathfrak{X}}U/U \leqslant \mathfrak{G}/U \in \mathcal{C}
$$
and~$\mathcal{C}$ is closed under taking subgroups. Since $\lambda$ acts identically on~$X$, it follows from the~equalities
\[
\tilde V \cap X = U \cap \tilde{\mathfrak{X}} \cap X = U \cap X = U \cap G \cap X = N \cap X = R
\]
that $V \cap X = R$. Therefore, $R \in \mathcal{C}_{\cap}^{*}(X, L, M, \varphi)$.
\end{proof}

Let $X$ be a~group, and~let $Y$ and~$Z$ be its subgroups. Recall that a~family~$\Omega$ of~normal subgroups of~$X$ is said to~be a~\emph{filtration} if $\bigcap_{N \in \Omega}N = 1$. A~filtration $\Omega$ is called

--\hspace{1ex}a~\emph{$Y$\nobreakdash-fil\-tra\-tion} if $\bigcap_{N \in \Omega}YN = 1$;

--\hspace{1ex}a~\emph{$(Y, Z)$\nobreakdash-fil\-tra\-tion} if it is a~$Y$\nobreakdash-fil\-tra\-tion and~a~$Z$\nobreakdash-fil\-tra\-tion.

\begin{eproposition}\label{ep402}
Suppose that $\mathcal{C}$ is a~class of~groups closed under taking subgroups, $(G, H, K, \varphi)$ is an~HNN-tuple, $L = H \cap K$, $M = L\varphi$, and~$X$ is a~subgroup of~$G$ containing~$L$ and~$M$. Then the~following statements hold.

\textup{1.}\hspace{1ex}If $\mathcal{C}_{\cap}^{*}(G, H, K, \varphi)$ is a~filtration, then $\mathcal{C}_{\cap}^{*}(X, L, M, \varphi)$ is also a~filtration.

\textup{2.}\hspace{1ex}Let $X$ coincide with~either $G$ or~$K$. If $\mathcal{C}_{\cap}^{*}(G, H, K, \varphi)$ is an~$(H, K)$\nobreakdash-fil\-tra\-tion, then $\mathcal{C}_{\cap}^{*}(X, L, M, \varphi)$ is an~$(L, M)$\nobreakdash-fil\-tra\-tion.
\end{eproposition}

\begin{proof}
{\parfillskip=0pt
By~Proposition~\ref{ep401}, if $N \in \mathcal{C}_{\cap}^{*}(G, H, K, \varphi)$, then $N \cap X \in \mathcal{C}_{\cap}^{*}(X, L, M, \varphi)$. Hence,
\[
\bigcap_{R \in \mathcal{C}_{\cap}^{*}(X, L, M, \varphi)}\kern-20pt{}R\kern10pt{} \leqslant \bigcap_{N \in \mathcal{C}_{\cap}^{*}(G, H, K, \varphi)}\kern-20pt{}(N \cap X),
\quad
\bigcap_{R \in \mathcal{C}_{\cap}^{*}(X, L, M, \varphi)}\kern-20pt{}RL\kern2pt{} \leqslant \bigcap_{N \in \mathcal{C}_{\cap}^{*}(G, H, K, \varphi)}\kern-20pt{}(N \cap X)L.
\]

}If $\mathcal{C}_{\cap}^{*}(G, H, K, \varphi)$ is a~filtration, then 
\[
1 = \bigcap_{N \in \mathcal{C}_{\cap}^{*}(G, H, K, \varphi)}\kern-20pt{}N\kern10pt{} = \bigcap_{N \in \mathcal{C}_{\cap}^{*}(G, H, K, \varphi)}\kern-20pt{}(N \cap X)
\]
and~therefore $\mathcal{C}_{\cap}^{*}(X, L, M, \varphi)$ is also a~filtration.

Let $\mathcal{C}_{\cap}^{*}(G, H, K, \varphi)$ be an~$(H, K)$\nobreakdash-fil\-tra\-tion. Then
\begin{multline*}
L \leqslant \bigcap_{R \in \mathcal{C}_{\cap}^{*}(X, L, M, \varphi)}\kern-20pt{}RL\kern2pt{} \leqslant \bigcap_{N \in \mathcal{C}_{\cap}^{*}(G, H, K, \varphi)}\kern-20pt{}NL\kern2pt{} = \bigcap_{N \in \mathcal{C}_{\cap}^{*}(G, H, K, \varphi)}\kern-20pt{}N(H \cap K) \\
\leqslant \bigg(\bigcap_{N \in \mathcal{C}_{\cap}^{*}(G, H, K, \varphi)}\kern-20pt{}NH\kern2pt{}\bigg) \cap \bigg(\bigcap_{N \in \mathcal{C}_{\cap}^{*}(G, H, K, \varphi)}\kern-20pt{}NK\kern2pt{}\bigg) = H \cap K = L
\end{multline*}
and~therefore $\mathcal{C}_{\cap}^{*}(X, L, M, \varphi)$ is an~$L$\nobreakdash-fil\-tra\-tion. To~prove that this family is an~$M$\nobreakdash-fil\-tra\-tion, we take an~arbitrary element $x \in X \setminus M$ and~indicate a~subgroup $R \in \mathcal{C}_{\cap}^{*}(X, L, M, \varphi)$ satisfying the~condition $x \notin MR$.

If $x \notin K$, then, because $\mathcal{C}_{\cap}^{*}(G, H, K, \varphi)$ is a~$K$\nobreakdash-fil\-tra\-tion, there exists a~subgroup $N \in\nolinebreak \mathcal{C}_{\cap}^{*}(G, H, K, \varphi)$ such that $x \notin KN$. It follows from~Proposition~\ref{ep401} and~the~inclusion $M(N \cap X) \leqslant KN$ that $N \cap X \in \mathcal{C}_{\cap}^{*}(X, L, M, \varphi)$ and~$x \notin M(N \cap X)$. Hence, $N \cap X$ is the~required subgroup.

If $x \in K$, we put $y = x\varphi^{-1}$. Since $x \notin M$, then $y \in H \setminus L$ and~therefore $y \notin K$. Let us consider two cases.
\smallskip

\textit{Case~1}. $X = G$.
\smallskip

Let $\mathfrak{X} = \mathrm{HNN}(G, L, M, \varphi)$. As~proved above, $\mathcal{C}_{\cap}^{*}(G, L, M, \varphi)$ is an~$L$\nobreakdash-fil\-tra\-tion. Therefore, $y \notin LR$ for~some subgroup $R \in \mathcal{C}_{\cap}^{*}(G, L, M, \varphi)$. By~the~definition of~the~family $\mathcal{C}_{\cap}^{*}(G, L, M, \varphi)$, there exists a~subgroup $U \in \mathcal{C}^{*}(\mathfrak{X})$ such that $R = U \cap G$. If $y \in LU$ and~$y = hu$ for~some $h \in L$, $u \in U$, then $u = h^{-1}y \in U \cap G = R$ and~we get the~inclusion $y \in LR$, which contradicts the~choice of~$R$. Therefore, $y \notin LU$~and
$$
x = y\varphi = t^{-1}yt \notin t^{-1}LUt.
$$
Since $U$ is normal in~$\mathfrak{X}$, then $t^{-1}LUt\kern-1pt{} =\kern-1pt{} MU$. Hence, $x\kern-1pt{} \notin\kern-1pt{} MU$ and~$x\kern-1pt{} \notin\kern-1pt{} MR$ because~$R\kern-1pt{} \leqslant\nolinebreak\kern-1pt{} U$. Thus, $R$ is the~required subgroup.
\smallskip

\textit{Case~2}. $X = K$.
\smallskip

Since $\mathcal{C}_{\cap}^{*}(G, H, K, \varphi)$ is a~$K$\nobreakdash-fil\-tra\-tion, there exists a~subgroup $N \in \mathcal{C}_{\cap}^{*}(G, H, K, \varphi)$ such that $y \notin KN$. By~Proposition~\ref{ep304}, $N$ is $(H, K, \varphi)$-com\-pat\-i\-ble and~therefore
$$
(M(N \cap K))\varphi^{-1} = L(N \cap H) \leqslant KN.
$$
Hence, $y \notin (M(N \cap K))\varphi^{-1}$ and~$x = y\varphi \notin M(N \cap K)$. It remains to~note that, by~Proposition~\ref{ep401}, $N \cap K \in \mathcal{C}_{\cap}^{*}(K, L, M, \varphi)$ and~thus $N \cap K$ is the~required subgroup.
\end{proof}

Let $(G, H, K, \varphi)$ be an~HNN-tuple. Following~\cite{Moldavanskii2002}, we say that a~subgroup $N \leqslant G$ is an~\emph{ascend} of~a~subgroup $R \leqslant K$ if $K \cap N = R$.

Let $H$ and~$K$ lie in~the~center of~$G$. We call an~ascend $N \leqslant G$ of~a~subgroup $R \leqslant K$ \emph{canonical} if $H \cap N = R\varphi^{-1}$ and~$HK \cap N = (R\varphi^{-1})R$. It should be noted that a~canonical ascend $N$ is always an~$(H, K, \varphi)$-com\-pat\-i\-ble subgroup because $(H \cap N)\varphi = R = K \cap N$.{\parfillskip=0pt{}\par}

\begin{eproposition}\label{ep403}
If $\mathcal{C}$ is a~class of~groups, $(G, H, K, \varphi)$ is an~HNN-tuple, $H$~and~$K$ lie in~the~center of~$G$, $L = H \cap K$, and~$M = L\varphi$, then the~following statements hold.

\textup{1.}\hspace{1ex}Suppose that $\mathcal{C}$ is closed under taking subgroups, quotient groups, and~extensions, $G$ is $\mathcal{C}$\nobreakdash-reg\-u\-lar with~respect to~the~subgroup $P = HK$, and~$R \in \mathcal{C}^{*}(K, L, M, \varphi)$. Then there exists a~canonical ascend $N \in \mathcal{C}^{*}(G, H, K, \varphi)$ of~$R$. Moreover, if $X$ is a~normal subgroup of~$G$ such that $P \leqslant X$ and~$G/X$ is residually a~$\mathcal{C}$\nobreakdash-group, then, for~any finite set $S \subseteq G \setminus X$, there exists a~canonical ascend $N \in \mathcal{C}^{*}(G, H, K, \varphi)$ of~$R$ satisfying the~condition $S \cap XN = \varnothing$.

\textup{2.}\hspace{1ex}If $R \in \mathcal{C}_{r+1}^{*}(K, L, M, \varphi)$ for~some $r \geqslant 0$ and~$N \in \mathcal{C}^{*}(G, H, K, \varphi)$ is a~canonical ascend of~$R$, then $N \in \mathcal{C}_{r}^{*}(G, H, K, \varphi)$.
\end{eproposition}

\begin{proof}
1.\hspace{1ex}If $X = G$ and~$S = \varnothing$, then $P \leqslant X$, $G/X$ is residually a~$\mathcal{C}$\nobreakdash-group, and~$S \subseteq G \setminus X$. Therefore, it suffices to~prove only the~second part of~Statement~1. We put $Q = R\varphi^{-1}$ and~show that $H \cap QR = Q$, $K \cap QR = R$.

Since $L, Q \leqslant H$, $\varphi$ is injective, and~$R$ is $(L, M, \varphi)$-com\-pat\-i\-ble, then
$$
L \cap Q = M\varphi^{-1} \cap R\varphi^{-1} = (M \cap R)\varphi^{-1} = L \cap R.
$$
If $g \in H \cap QR$ and~$g = xy$ for~some $x \in Q$, $y \in R$, then $g, x \in H$,
$$
y \in H \cap R = H \cap K \cap R = L \cap R = L \cap Q,
$$
and~therefore $g \in Q$. Hence, $H \cap QR \subseteq Q$. The~relation $K \cap QR \subseteq R$ is verified in~the~same way. Since the~opposite inclusions are obvious, the~required equalities are proved.

The~quotient group~$P/QR$ is an~extension of~$HR/QR$ by~$P/HR$. We have
\begin{gather*}
HR/QR \cong H/Q(R \cap H) \cong (H/Q)/(Q(R \cap H)/Q),\\
P/HR = HK/HR \cong K/R(H \cap K) \cong (K/R)/(R(H \cap K)/R).
\end{gather*}

Since $R \in \mathcal{C}^{*}(K, L, M, \varphi)$, then $K/R \in \mathcal{C}$ and~therefore $H/Q \in \mathcal{C}$. Hence, $P/QR \in \mathcal{C}$ because $\mathcal{C}$ is closed under taking quotient groups and~extensions. Thus, using the~$\mathcal{C}$\nobreakdash-reg\-u\-larity of~$G$ with~respect to~$P$, we can find a~subgroup $U \in \mathcal{C}^{*}(G)$ such that $U \cap P = QR$.

By~the~hypothesis of~the~proposition, $G/X$ is residually a~$\mathcal{C}$\nobreakdash-group and~$\mathcal{C}$ is closed under taking subgroups and~extensions. Hence, $\mathcal{C}$ is also closed under taking direct products of~a~finite number of~factors and,~by~Proposition~\ref{ep301}, there exists a~subgroup $V/X \in\nolinebreak \mathcal{C}^{*}(G/X)$ such that $SX/X \cap V/X = \varnothing$. Then $V \in \mathcal{C}^{*}(G)$ and~$S \cap V = \varnothing$.

Let $N = U \cap V$. Then $S \cap XN \subseteq S \cap XV = S \cap V = \varnothing$ and,~again by~Proposition~\ref{ep301}, $N \in \mathcal{C}^{*}(G)$. Because $P \leqslant X \leqslant V$, we have
\begin{gather*}
P \cap N = P \cap U \cap V = P \cap U = QR,\\
H \cap N = H \cap P \cap N = H \cap QR = Q,\\
K \cap N = K \cap P \cap N = K \cap QR = R.
\end{gather*}

Since $Q\varphi = R$, it follows from~these equalities that $N$ is an~$(H, K, \varphi)$-com\-pat\-i\-ble subgroup, which is the~required canonical ascend of~$R$.
\smallskip

2.\hspace{1ex}The~relation $R \in \mathcal{C}_{r+1}^{*}(K, L, M, \varphi)$ means that there exists a~number $n \geqslant r+3$ such that a~cyclic group~$C_{n}$ of~order~$n$ belongs to~$\mathcal{C}$ and~$n$ is admissible \pagebreak for~the~HNN-tuple $(K/R, LR/R, MR/R, \varphi_{R})$ with~the~reserve~$r+1$. Since $K \cap N = R$, the~map of~$K/R$ into~$G/N$ taking a~coset~$kR$ ($k \in K$) to~$kN$ is an~isomorphism of~$K/R$ onto~the~subgroup~$KN/N$ of~$G/N$. Under this isomorphism, the~subgroups~$LR/R$ and~$MR/R$ are mapped onto~the~subgroups~$LN/N$ and~$MN/N$, and~the~isomorphism~$\varphi_{R}$ corresponds to~the~isomorphism~$\varphi_{N}$. Therefore, $n$ is admissible with~the~reserve~$r+1$ for~the~HNN-tuple $(KN/N, LN/N, MN/N, \varphi_{N})$. Let us show that $HN/N \cap KN/N = LN/N$.

If $gN \in HN/N \cap KN/N$, then $g \in HN \cap KN$ and~$g = hx = ky$ for~some $h \in H$, $k \in K$, and~$x, y \in N$. Hence, $h^{-1}k = xy^{-1} \in HK \cap N$. Since $N$ is a~canonical ascend of~$R$, then $HK \cap N = QR$ (where, as~above, $Q = R\varphi^{-1}$) and~therefore $h^{-1}k = h_{1}k_{1}$ for~some $h_{1} \in Q$, $k_{1} \in R$. Thus, $hh_{1}^{\vphantom{1}} = kk_{1}^{-1} \in H \cap K = L$ and~$h \in h_{1}^{-1}L$. Since $h_{1} \in Q \leqslant N$, then $g = hx \in LN$ and~$gN \in LN/N$. Therefore, $HN/N \cap KN/N \subseteq LN/N$ and,~because the~opposite inclusion is obvious, the~required equality is proved.

Thus, $HN/N \cap KN/N = LN/N$ and~$(LN/N)\varphi_{N} = MN/N$ by~the~definition of~$\varphi_{N}$. It~follows from~these equalities and~Proposition~\ref{ep202} that $n$ is admissible for~the~HNN-tuple $(G/N, HN/N, KN/N, \varphi_{N})$ with~the~reserve~$r$. Since $C_{n} \in \mathcal{C}$, we have $N \in \mathcal{C}_{r}^{*}(G, H, K, \varphi)$, as~required.
\end{proof}

Let $X$ be a~group, and~let $Y$, $Z$ be its subgroups. We call a~family~$\Omega$ of~normal subgroups of~$X$ a~\emph{strong $(Y, Z)$-fil\-tra\-tion} if, for~any finite subset~$S$ of~$X$, there exists a~subgroup $N \in \Omega$ such that, for~each $x \in S$, the~following statements hold:
\smallskip

1)\hspace{1ex}if $x \ne 1$, then $x \notin N$;

2)\hspace{1ex}if $x \notin Y$, then $x \notin YN$;

3)\hspace{1ex}if $x \notin Z$, then $x \notin ZN$.
\smallskip

It is easy to~see that any strong $(Y, Z)$-fil\-tra\-tion is a~$(Y, Z)$-fil\-tra\-tion and,~if $\Omega$ is closed under taking finite intersections of~subgroups, then any $(Y, Z)$-fil\-tra\-tion is a~strong $(Y, Z)$-fil\-tra\-tion.

\begin{eproposition}\label{ep404}
Suppose that $\mathcal{C}$ is a~class of~groups closed under taking subgroups, quotient groups, and~extensions, $(G, H, K, \varphi)$ is an~HNN-tuple, $H$~and~$K$ lie in~the~center of~$G$ and~are $\pi^{\prime}$\nobreakdash-iso\-lat\-ed in~this group for~some set of~primes~$\pi$, $L = H \cap K$, $M = L\varphi$, and~$G$ is $\mathcal{C}$\nobreakdash-reg\-u\-lar with~respect to~the~subgroup $P = HK$. Suppose also that there exists a~normal subgroup~$X$ of~$G$ such that $P \leqslant X$, $X/P$ is a~periodic $\pi^{\prime}$\nobreakdash-group, and~$G/X$ is residually a~$\mathcal{C}$\nobreakdash-group. Then the~following statements hold. 

\textup{1.}\hspace{1ex}If $\mathcal{C}^{*}(K, L, M, \varphi)$ is a~strong $(L, M)$-fil\-tra\-tion, then $\mathcal{C}^{*}(G, H, K, \varphi)$ is a~strong $(H, K)$-fil\-tra\-tion.

\textup{2.}\hspace{1ex}If, for~some $r \geqslant 0$, $\mathcal{C}_{r+1}^{*}(K, L, M, \varphi)$ is a~strong $(L, M)$-fil\-tra\-tion, then $\mathcal{C}_{r}^{*}(G, H, K, \varphi)$ is a~strong $(H, K)$-fil\-tra\-tion.
\end{eproposition}

\begin{proof}
We prove Statements~1 and~2 simultaneously. Suppose that $S$ is a~finite subset of~$G$ and~$g \in S$. Then:
\smallskip

1)\hspace{1ex}if $g \in K \setminus H$, then $g \notin L$;

2)\hspace{1ex}if $g \in H \setminus K$, then $g\varphi \notin M$;

3)\hspace{1ex}if $g \in P \setminus (H \cup K)$ and~$g = hk$, where $h \in H$, $k \in K$, then $k \notin L$, $h\varphi \notin M$;

4)\hspace{1ex}if $g \in X \setminus P$, $q$ is the~order of~$g$ modulo~$P$, and~$g^{q} = hk$, where $h \in H$, $k \in K$, then $k \notin L$, $h\varphi \notin M$.
\smallskip

Indeed, if $g \in K \setminus H$, then $g \notin H \cap K = L$; if $g \in H \setminus K$, then $g \notin H \cap K = L$ and~therefore $g\varphi \notin M$. If $g \in P \setminus (H \cup K)$ and~$g = hk$ for~some $h \in H$, $k \in K$, then $h \notin K$, $k \notin H$, and,~as~above, $k \notin L$, $h\varphi \notin M$. Suppose that $g \in X \setminus P$, $q$ is the~order of~$g$ modulo~$P$, and~$g^{q} = hk$, where $h \in H$, $k \in K$. Since $g \notin H \cup K$, $H$~and~$K$ are $\pi^{\prime}$\nobreakdash-iso\-lat\-ed in~$G$, and,~by~the~hypothesis of~the~proposition, $q$ is a~$\pi^{\prime}$\nobreakdash-num\-ber, then $hk \notin H \cup K$. Hence, using the~above argument, we get $k \notin L$, $h\varphi \notin M$.

Thus, because the~family $\mathcal{C}^{*}(K, L, M, \varphi)$ or~$\mathcal{C}_{r+1}^{*}(K, L, M, \varphi)$ is a~strong $(L, M)$-fil\-tra\-tion, we can choose a~subgroup~$R$ from~this family so that, for~each $g \in S$, the~following conditions hold: 
\smallskip

1)\hspace{1ex}if $g \in L$ and~$g \ne 1$, then $g \notin R$;

2)\hspace{1ex}if $g \in K \setminus H$, then $g \notin LR$;

3)\hspace{1ex}if $g \in H \setminus K$, then $g\varphi \notin MR$;

4)\hspace{1ex}if $g \in P \setminus (H \cup K)$ and~$g = hk$, where $h \in H$, $k \in K$, then $k \notin LR$, $h\varphi \notin MR$;

5)\hspace{1ex}if $g \in X \setminus P$, $q$ is the~order of~$g$ modulo~$P$, and~$g^{q} = hk$, where $h \in H$, $k \in K$, then $k \notin LR$, $h\varphi \notin MR$.
\smallskip

It is sufficient for~us that, for~each $g \in S$, the~fourth and~fifth conditions hold for~some fixed choice of~$h$ and~$k$. However, if $h_{1}k_{1} = h_{2}k_{2}$ for~some $h_{1}, h_{2} \in H$, $k_{1}, k_{2} \in K$, then $h_{1}^{-1}h_{2}^{\vphantom{1}} = k_{1}^{\vphantom{1}}k_{2}^{-1} \in L$. Therefore, if $k_{1} \notin LR$ and~$h_{1}\varphi \notin MR$, then $k_{2} \notin LR$ and~$h_{2}\varphi \notin MR$.

Let $S_{1} = S \setminus X$, and~let $Q = R\varphi^{-1}$. By~Proposition~\ref{ep403}, there exists a~canonical ascend $N \in \mathcal{C}^{*}(G, H, K, \varphi)$ of~$R$ such that $S_{1} \cap XN = \varnothing$ and,~if $R \in \mathcal{C}_{r+1}^{*}(K, L, M, \varphi)$, then $N \in \mathcal{C}_{r}^{*}(G, H, K, \varphi)$. To~complete the~proof, it suffices to~show that, for~each $g \in S$, the~following statements hold:
\smallskip

1)\hspace{1ex}if $g \ne 1$, then $g \notin N$;

2)\hspace{1ex}if $g \notin H$, then $g \notin HN$;

3)\hspace{1ex}if $g \notin K$, then $g \notin KN$.
\smallskip

Let us consider several cases.
\smallskip

\textit{Case~1:} $g \in L$ and~$g \ne 1$.
\smallskip

Since $R = N \cap K$, $g \in K$, and~$g \notin R$ by~the~choice of~$R$, then $g \notin N$.
\smallskip

\textit{Case~2:} $g \in K \setminus H$.
\smallskip

If $g \in HN$ and~$g = hx$ for~some $h \in H$, $x \in N$, then $x = h^{-1}g \in N \cap P$. Since $N$ is a~canonical ascend of~$R$, then $N \cap P = QR$ and~$x = h_{1}k$ for~some $h_{1} \in Q$, $k \in R$. Hence, $gk^{-1} = hh_{1} \in H \cap K = L$, $g \in LR$, and~we get a~contradiction with~the~choice of~$R$.
\smallskip

\textit{Case~3:} $g \in H \setminus K$.
\smallskip

As~in~the~previous case, if $g \in KN$ and~$g = kx$ for~some $k \in K$, $x \in N$, then $x = k^{-1}g \in N \cap P = RQ$ and~therefore $x = k_{1}h$ for~some $h \in Q$, $k_{1} \in R$. Hence, $gh^{-1} = kk_{1} \in L$, $g = (kk_{1})h \in LQ$, $g\varphi \in MR$, and~we again get a~contradiction with~the~choice of~$R$.
\smallskip

\textit{Case~4:} $g \in P \setminus (H \cup K)$.
\smallskip

Let $g = hk$, where $h \in H$, $k \in K$. If $g \in HN$ and~$g = h_{1}x$ for~some $h_{1} \in H$, $x \in N$, then it follows from~the~equality $hk = h_{1}x$ that $x \in P$ and~therefore $x \in N \cap P = QR$. Hence, $x = h_{2}k_{2}$ for~some $h_{2} \in Q$, $k_{2} \in R$, $hk = h_{1}h_{2}k_{2}$, and~$kk_{2}^{-1} = h^{-1}h_{1}h_{2} \in L$. Thus, we get the~relations $k = (h^{-1}h_{1}h_{2})k_{2} \in LR$, which contradict the~choice of~$R$.

Similarly, if $g \in KN$ and~$g = k_{1}x$ for~some $k_{1} \in K$, $x \in N$, then again $x \in N \cap P = QR$ and~$x = h_{2}k_{2}$ for~some $h_{2} \in Q$, $k_{2} \in R$. Hence, $hk = k_{1}h_{2}k_{2}$ and,~since $[H, K] = 1$, then $hh_{2}^{-1} = k^{-1}k_{1}k_{2} \in L$, $h = (k^{-1}k_{1}k_{2})h_{2} \in LQ$, and~$h\varphi \in MR$. The~last inclusion again contradicts the~choice of~$R$.
\smallskip

\textit{Case~5:} $g \in X \setminus P$.
\smallskip

As~already proved above, it follows from~the~relation $g \in X \setminus P$ that $g^{q} \notin H \cup K$, where $q$ is the~order of~$g$ modulo~$P$. Hence, using the~same argument as~in~Case~4, we get $g^{q} \notin HN \cup KN$. Then $g \notin HN \cup KN$, as~required.
\smallskip

\textit{Case~6:} $g \notin X$.
\smallskip

Since $g \notin X$, then $g \in S_{1}$, $g \notin XN$ by~the~choice of~$N$, and~$g \notin HN \cup KN$ because $H \cup K \subseteq X$.
\end{proof}

\section{Necessary and~sufficient conditions for~the~$\mathcal{C}$-residuality of~HNN-extensions}\label{esec05}

Throughout this section, it is assumed that $(G, H, K, \varphi)$ is an~HNN-tuple and~$\mathfrak{G} = \mathrm{HNN}(G, H, K, \varphi)$. Let us recall that any element $g \in \mathfrak{G}$ can be written in~a~\emph{reduced form}:
$$
g = g_{0}t^{\varepsilon_1}g_{1}\ldots t^{\varepsilon_n}g_{n},
$$
where $g_{i} \in G$, $\varepsilon_{i} = \pm 1$, and~if $-\varepsilon_{i} = 1 = \varepsilon_{i+1}$, then $g_{i} \notin H$, if $\varepsilon_{i} = 1 = -\varepsilon_{i+1}$, then $g_{i} \notin K$. Britton's lemma (see, e.~g., \cite[Ch.~IV, Section~2]{LyndonSchupp1977}) says that an~element of~$\mathfrak{G}$ is non-triv\-i\-al if it has a~reduced form containing at~least one letter~$t$ or~$t^{-1}$. As~a~consequence, all the~reduced forms of~an~element $g \in \mathfrak{G}$ have the~same number of~occurrences of~the~letters~$t$ and~$t^{-1}$. This number is called the~\emph{length} of~$g$ and~is denoted in~this paper by~$|g|$.

\begin{eproposition}\label{ep501}
If $\mathcal{C}$ is a~root class of~groups, then the~following statements hold.

\textup{1.}\hspace{1ex}Every free group is residually a~$\mathcal{C}$\nobreakdash-group~\textup{\cite[Theorem~1]{AzarovTieudjo2002}}.

\textup{2.}\hspace{1ex}Any extension of~a~residually $\mathcal{C}$\nobreakdash-group by~a~$\mathcal{C}$\nobreakdash-group is residually a~$\mathcal{C}$\nobreakdash-group~\textup{\cite[Lemma~1.5]{Gruenberg1957}}.
\end{eproposition}

The~following proposition is a~generalization of~Theorem~4.2 from~\cite{BaumslagBTretkoff1978} and~is partially proved in~\cite{Tumanova2014}.

\begin{eproposition}\label{ep502}
Let $\mathcal{C}$ be a~class of~groups.

\textup{1.}\hspace{1ex}If $\mathfrak{G}$ is residually a~$\mathcal{C}$\nobreakdash-group, then $\mathcal{C}_{\cap}^{*}(G, H, K, \varphi)$ is a~filtration.

\textup{2.}\hspace{1ex}If $\mathfrak{G}$ is residually a~$\mathcal{C}$\nobreakdash-group, $H$~and~$K$ are proper central subgroups of~$G$, then $\mathcal{C}_{\cap}^{*}(G, H, K, \varphi)$ is an~$(H, K)$\nobreakdash-fil\-tra\-tion.

\textup{3.}\hspace{1ex}If $\mathcal{C}$ is a~root class of~groups and~$\mathcal{C}_{\cap}^{*}(G, H, K, \varphi)$ is an~$(H, K)$\nobreakdash-fil\-tra\-tion, then $\mathfrak{G}$ is residually a~$\mathcal{C}$\nobreakdash-group.
\end{eproposition}

\textit{Proof.} 1.\hspace{1ex}Since $\mathfrak{G}$ is residually a~$\mathcal{C}$\nobreakdash-group, then
\[
1 = \bigcap_{U \in \mathcal{C}^{*}(\mathfrak{G})}\kern-5pt{}U = \bigcap_{U \in \mathcal{C}^{*}(\mathfrak{G})}\kern-5pt{}(U \cap G) = \bigcap_{N \in \mathcal{C}_{\cap}^{*}(G, H, K, \varphi)}\kern-20pt{}N,
\]
as~required.
\smallskip

2.\hspace{1ex}Suppose that $\mathcal{C}_{\cap}^{*}(G, H, K, \varphi)$ is not an~$H$\nobreakdash-fil\-tra\-tion~and
\[
g_{1} \in \bigg(\bigcap_{N \in \mathcal{C}_{\cap}^{*}(G, H, K, \varphi)}\kern-20pt{}HN\bigg) \setminus H.
\]
Because $K \ne G$, we can take an~element $g_{2} \in G \setminus K$ and~put $g = [t^{-1}g_{1}t, g_{2}]$. Then $|g| = 4$ and~therefore $g \ne 1$. Since $\mathfrak{G}$ is residually a~$\mathcal{C}$\nobreakdash-group, there exists a~subgroup $U \in \mathcal{C}^{*}(\mathfrak{G})$ such that $g \notin U$. By~the~choice of~$g_{1}$, the~inclusion $g_{1} \in H(U \cap G)$ holds, and~therefore $g_{1} \equiv h \pmod U$ for~some $h \in H$. Hence,
$$
g \equiv [t^{-1}ht, g_{2}] = [h\varphi, g_{2}] \pmod U.
$$
Since $h\varphi \in K$ and~$K$ is central in~$G$, then $[h\varphi, g_{2}] = 1$, and~we get a~contradiction with the~choice of~$U$. The~fact that $\mathcal{C}_{\cap}^{*}(G, H, K, \varphi)$ is a~$K$\nobreakdash-fil\-tra\-tion is proved in~a~similar~way.{\parfillskip=0pt\par}
\smallskip

3.\hspace{1ex}Let $g \in \mathfrak{G} \setminus \{1\}$. We show that there exists a~subgroup $N \in \mathcal{C}_{\cap}^{*}(G, H, K, \varphi)$ such that $g\rho_{N} \ne 1$.

If $g \in G$, then, because $\mathcal{C}_{\cap}^{*}(G, H, K, \varphi)$ is a~filtration, $g$ does not belong to~some subgroup $N \in \mathcal{C}_{\cap}^{*}(G, H, K, \varphi)$. Since $\rho_{N}$ extends the~natural homomorphism $G \to G/N$, then $g\rho_{N} \ne 1$ and~therefore $N$ is the~required subgroup.

Suppose that $g\kern-1.5pt{} \notin\kern-1.5pt{} G$ and~$g\kern-1.5pt{} =\kern-1.5pt{} g_{0}t^{\varepsilon_1}g_{1}\kern-1pt{}\ldots t^{\varepsilon_n}g_{n}$, where $g_{i}\kern-1.5pt{} \in\kern-1.5pt{} G$, $\varepsilon_{i}\kern-1.5pt{} =\kern-1.5pt{} \pm 1$, and~if $-\varepsilon_{i}\kern-1.5pt{} =\nolinebreak\kern-1.5pt{} 1\kern-1.5pt{} =\nolinebreak\kern-1.5pt{} \varepsilon_{i+1}$, then $g_{i} \notin H$, if $\varepsilon_{i} = 1 = -\varepsilon_{i+1}$, then $g_{i} \notin K$. It should be noted that, since $g \notin G$, then $n \geqslant 1$. For~each $i \in \{0,\, 1,\, \ldots,\, n\}$, we define a~subgroup $N_{i} \in \mathcal{C}_{\cap}^{*}(G, H, K, \varphi)$ as~follows.{\parfillskip=0pt\par}

If~$1 \leqslant i \leqslant n-1$ and~$-\varepsilon_{i} = 1 = \varepsilon_{i+1}$, then $g_{i} \notin H$ and,~since $\mathcal{C}_{\cap}^{*}(G, H, K, \varphi)$ is an~$(H, K)$-fil\-tra\-tion, we can find a~subgroup $N_{i} \in \mathcal{C}_{\cap}^{*}(G, H, K, \varphi)$ such that $g_{i} \notin HN_{i}$. Similarly, if $1 \leqslant i \leqslant n-1$ and~$\varepsilon_{i} = 1 = -\varepsilon_{i+1}$, we choose a~subgroup $N_{i} \in \mathcal{C}_{\cap}^{*}(G, H, K, \varphi)$ so that $g_{i} \notin KN_{i}$. For~all other $i \in \{0,\, 1,\, \ldots,\, n\}$, we put $N_{i} = G$. The~last choice is possible because $\mathcal{C}$ is root and~therefore the~trivial group belongs to~$\mathcal{C}$, $\mathfrak{G} \in \mathcal{C}^{*}(\mathfrak{G})$, and~$G \in \mathcal{C}_{\cap}^{*}(G, H, K, \varphi)$.

Let
$$
N = \bigcap_{i=0}^{n}N_{i}.
$$

Then $N \in \mathcal{C}_{\cap}^{*}(G, H, K, \varphi)$ by~Proposition~\ref{ep302} and,~for~any $i \in \{1,\, 2,\, \ldots,\, n-1\}$,\linebreak if $-\varepsilon_{i} = 1 = \varepsilon_{i+1}$, then $g_{i} \notin HN$, if $\varepsilon_{i} = 1 = -\varepsilon_{i+1}$, then $g_{i} \notin KN$. Hence,
$$
g\rho_{N} = (g_{0}N)t^{\varepsilon_1}(g_{1}N)\ldots t^{\varepsilon_n}(g_{n}N)
$$
is a~reduced form of~$g\rho_{N}$ of~length $n \geqslant 1$ and,~in~particular, $g\rho_{N} \ne 1$. Thus, $N$ is the~required subgroup.

Let us now show that $\mathfrak{G}_{N} = \mathrm{HNN}(G/N, HN/N, KN/N, \varphi_{N})$ is residually a~$\mathcal{C}$\nobreakdash-group and~therefore $\rho_{N}$ can be extended to~a~homomorphism of~$\mathfrak{G}$ onto~a~group from~$\mathcal{C}$ taking~$g$ to~a~non-triv\-i\-al element. Suppose that $U \in \mathcal{C}^{*}(\mathfrak{G})$ is a~subgroup satisfying the~equality $N = U \cap G$ and~$\sigma\colon \mathfrak{G}_{N} \to \mathfrak{G}/U$ is the~mapping defined by~the~rule: $(x\rho_{N})\sigma = xU$ ($x \in \mathfrak{G}$). Since the~kernel of~$\rho_{N}$ coincides with~the~normal closure of~$N$ in~$\mathfrak{G}$, it is contained in~$U$ and~therefore $\sigma$ is well defined. It is also easy to~see that $\sigma$ is a~surjective homomorphism. Since $G\rho_{N} = G/N$ and~$N = U \cap G$, then $\ker\sigma \cap G/N = 1$. Hence, $\ker\sigma$ is free (see, \mbox{e.\;g.,~\cite{Cohen1974}}) and~$\mathfrak{G}_{N}$ is an~extension of~the~free group~$\ker\sigma$ by~the~$\mathcal{C}$\nobreakdash-group~$\mathfrak{G}/U$. Such an~extension is residually a~$\mathcal{C}$\nobreakdash-group by~Proposition~\ref{ep501}.

\begin{eproposition}\label{ep503}
Suppose that $\mathcal{C}$ is a~class of~groups consisting only of~periodic groups, $H$ is a~central subgroup of~$G$, and~$H \leqslant K \ne G$. If $\mathfrak{G}$ is residually a~$\mathcal{C}$\nobreakdash-group, then $H = K$.
\end{eproposition}

\begin{proof}
Suppose that $H \ne K$ and~$k \in K \setminus H$. Suppose also that $g \in G \setminus K$ and~$x = [t^{-1}kt, g]$. Then $|x| = 4$ and~therefore $x \ne 1$. Since $\mathfrak{G}$ is residually a~$\mathcal{C}$\nobreakdash-group, there exists a~subgroup $U \in \mathcal{C}^{*}(\mathfrak{G})$ such that $x \notin U$. Because $\mathcal{C}$ consists only of~periodic groups, $t$ has a~finite order modulo~$U$ and~therefore $t^{-1} \equiv t^{m} \pmod U$ for~some $m > 0$. It follows from~the~last relation and~the~inclusion $H \leqslant K$ that
$$
t^{-1}kt \equiv t^{m}kt^{-m} = k\varphi^{-m} \in H \pmod U.
$$
Since $H$ is central in~$G$, then $[k\varphi^{-m}, g] = 1$. Thus, $x \in U$, and~we get a~contradiction with~the~choice of~$U$.
\end{proof}

\begin{eproposition}\label{ep504}
Suppose that $\mathcal{C}$ is a~class of~groups closed under taking subgroups and~direct products of~a~finite number of~factors. Suppose also that $G$ is residually a~$\mathcal{C}$\nobreakdash-group and~there exists a~subgroup $Q \in \mathcal{C}^{*}(G)$ satisfying at~least one of~the~following conditions: 

$(\alpha)$\hspace{1ex}$H \cap Q = 1 = K \cap Q$,

$(\beta)$\hspace{1ex}$Q \leqslant H \cap K$ and~$Q\varphi = Q$.

\noindent
Then $H$ and~$K$ are $\mathcal{C}$\nobreakdash-sep\-a\-ra\-ble in~$G$.
\end{eproposition}

\begin{proof}
We take an~element $g \in G \setminus H$ and~find a~subgroup $N \in \mathcal{C}^{*}(G)$ such that $g \notin HN$.

If $Q \leqslant H$ or~$g \notin HQ$, then $Q$ is the~required subgroup. Therefore, we can assume that $H \cap Q = 1$ and~$g = hx$ for~some $h \in H$, $x \in Q$. Since $g \notin H$, then $x \ne 1$ and,~because $G$ is residually a~$\mathcal{C}$\nobreakdash-group, there exists a~subgroup $M \in \mathcal{C}^{*}(G)$ that does not contain~$x$. We put $N = M \cap Q$. Then $N \in \mathcal{C}^{*}(G)$ by~Proposition~\ref{ep301}. If $g = h_{1}^{\vphantom{1}}x_{1}^{\vphantom{1}}$ for~some $h_{1}^{\vphantom{1}} \in H$, $x_{1}^{\vphantom{1}} \in N$, then $xx_{1}^{-1} = h^{-1}_{\vphantom{1}}h_{1}^{\vphantom{1}} \in H \cap Q = 1$. Hence, $x = x_{1}^{\vphantom{1}} \in N \leqslant M$, and~we get a~contradiction with~the~choice of~$M$. Therefore, $g \notin HN$ and~$N$ is the~required subgroup.

Thus, $H$ is $\mathcal{C}$\nobreakdash-sep\-a\-ra\-ble in~$G$. The~$\mathcal{C}$\nobreakdash-separability of~$K$ is proved in~a~similar way.
\end{proof}

\begin{eproposition}\label{ep505}
Suppose that $\mathcal{C}$ is a~root class of~groups containing infinite groups, $G$ is residually a~$\mathcal{C}$\nobreakdash-group, $H$~and~$K$ are central subgroups of~$G$. Suppose also that there exists a~subgroup $Q \in \mathcal{C}^{*}(G)$ satisfying at~least one of~the~following conditions:

$(\alpha)$\hspace{1ex}$H \cap Q = 1 = K \cap Q$,

$(\beta)$\hspace{1ex}$Q \leqslant H \cap K$ and~$Q\varphi = Q$.

\noindent
Then $\mathcal{C}^{*}(G, H, K, \varphi)$ is an~$(H, K)$-fil\-tra\-tion.
\end{eproposition}

\begin{proof}
We only need to~show that every subgroup from~$\mathcal{C}^{*}(G)$ contains some subgroup from~$\mathcal{C}^{*}(G, H, K, \varphi)$. Since $G$ is residually a~$\mathcal{C}$\nobreakdash-group and~$H$, $K$ are $\mathcal{C}$\nobreakdash-sep\-a\-ra\-ble in~$G$ by~Proposition~\ref{ep504}, then it will follow from~this fact that $\mathcal{C}^{*}(G, H, K, \varphi)$ is an~$(H, K)$-fil\-tra\-tion.

So, let $L \in \mathcal{C}^{*}(G)$. By~Proposition~\ref{ep301}, the~subgroup $M = L \cap Q$ belongs to~$\mathcal{C}^{*}(G)$. If~$(\alpha)$ holds, then
\[
(H \cap M)\varphi = (H \cap L \cap Q)\varphi = 1 = K \cap L \cap Q = K \cap M
\]
and~therefore $M \in \mathcal{C}^{*}(G, H, K, \varphi)$. If $(\beta)$ holds, then the~subgroup
$$
N = \bigcap_{i \in \mathbb{Z}}M\varphi^{i}
$$
belongs to~the~same family.

Indeed, $N$ lies in~the~center of~$G$, is $\varphi$\nobreakdash-in\-var\-i\-ant, and~therefore is $(H, K, \varphi)$\nobreakdash-com\-pat\-i\-ble. Since $Q/M \leqslant G/M$ and~$\mathcal{C}$ is closed under taking subgroups, then $M \in \mathcal{C}^{*}(Q)$. Because $Q$ is $\varphi$\nobreakdash-in\-var\-i\-ant, the~restriction of~$\varphi$ to~this subgroup is its automorphism and~therefore $M\varphi^{i} \in \mathcal{C}^{*}(Q)$. Hence, $Q/N$ is embedded into~the~Cartesian product of~countably many isomorphic $\mathcal{C}$\nobreakdash-groups~$Q/M\varphi^{i}$. Since $\mathcal{C}$ is root and~contains infinite groups, then it also contains the~indicated Cartesian product, its subgroup~$Q/N$, and~the~group~$G/N$, which is an~extension of~$Q/N$ by~the~$\mathcal{C}$\nobreakdash-group~$G/Q$. Therefore, $N \in \mathcal{C}^{*}(G)$.
\end{proof}

\begin{eproposition}\label{ep506}
Suppose that $\mathcal{C}$ is a~class of~groups closed under taking subgroups and~direct products of~a~finite number of~factors, $H$~and~$K$ are proper central subgroups of~finite index of~$G$. If $\mathfrak{G}$ is residually a~$\mathcal{C}$\nobreakdash-group, then there exists a~subgroup $Q \in \mathcal{C}^{*}(G)$ such that $Q \leqslant H \cap K$ and~$Q\varphi = Q$.
\end{eproposition}

\begin{proof}
Because $H$ and~$K$ are of~finite index in~$G$, the~subgroup $H \cap K$ has the~same property. Let $1 = g_{1},\, g_{2},\, \ldots,\, g_{n}$ be a~complete set of~cosets representatives of~this subgroup in~$G$, and~let $S = \{g_{2},\, \ldots,\, g_{n}\}$. Since $\mathfrak{G}$ is residually a~$\mathcal{C}$\nobreakdash-group, then, by~Proposition~\ref{ep502}, $\mathcal{C}_{\cap}^{*}(G, H, K, \varphi)$ is an~$(H, K)$\nobreakdash-fil\-tra\-tion. Hence,
\[
H \cap K \leqslant \bigcap_{N \in \mathcal{C}_{\cap}^{*}(G, H, K, \varphi)}\kern-20pt{}(H \cap K)N \leqslant \bigg(\bigcap_{N \in \mathcal{C}_{\cap}^{*}(G, H, K, \varphi)}\kern-20pt{}HN\bigg) \cap \bigg(\bigcap_{N \in \mathcal{C}_{\cap}^{*}(G, H, K, \varphi)}\kern-20pt{}KN\bigg) = H \cap K
\]
and~therefore, for~each $s \in S$, there exists a~subgroup $Q_{s} \in \mathcal{C}_{\cap}^{*}(G, H, K, \varphi)$ such that $s \notin (H \cap K)Q_{s}$. Let us show that
$$
Q = \bigcap_{s \in S}Q_{s}
$$
is the~required subgroup.

Indeed, if $g \in G \setminus (H \cap K)$ and~$g = xs$ for~suitable $s \in S$, $x \in H \cap K$, then $x^{-1}g = s \notin (H \cap K)Q_{s}$ and~therefore $x^{-1}g \notin (H \cap K)Q$. Hence, $g \notin (H \cap K)Q$ and,~because $g$ is chosen arbitrarily, $(H \cap K)Q \leqslant H \cap K$. Thus, $Q \leqslant H \cap K$. It remains to~note that $Q \in \mathcal{C}_{\cap}^{*}(G, H, K, \varphi)$ by~Proposition~\ref{ep302} and~$\mathcal{C}_{\cap}^{*}(G, H, K, \varphi) \subseteq \mathcal{C}^{*}(G, H, K, \varphi)$ by~Proposition~\ref{ep304}. Therefore, $Q \in \mathcal{C}^{*}(G)$ and~$Q\varphi = (Q \cap H)\varphi = Q \cap K = Q$.
\end{proof}
\pagebreak

We conclude this section with~two criteria for~the~root-class residuality of~split extensions.

\begin{eproposition}\label{ep507}
Suppose that $\mathcal{C}$ is a~root class of~groups consisting only of~periodic groups and~closed under taking quotient groups. Suppose also that $G$ is an~abelian group, $H = G = K$, and~$\Omega$ is the~family of~subgroups of~$G$ defined as~follows: $N \in \Omega$ if and~only if $N \in \mathcal{C}^{*}(G)$, $N\varphi = N$, and~the~automorphism~$\varphi_{N}$ of~the~group~$G/N$ induced by~the~automorphism~$\varphi$ has a~finite order, which is a~$\pi(\mathcal{C})$\nobreakdash-num\-ber. Then $\mathfrak{G}$ is residually a~$\mathcal{C}$\nobreakdash-group if and~only if $\bigcap_{N \in \Omega}N = 1$.
\end{eproposition}

\begin{proof}
Since $H = G = K$, then any filtration is an~$(H, K)$-fil\-tra\-tion. Therefore, by~Proposition~\ref{ep502}, $\mathfrak{G}$ is residually a~$\mathcal{C}$\nobreakdash-group if and~only if $\mathcal{C}_{\cap}^{*}(G, H, K, \varphi)$ is a~filtration. It remains to~note that, by~Proposition~\ref{ep305}, a~subgroup~$N$ belongs to~$\mathcal{C}_{\cap}^{*}(G, H, K, \varphi)$ if and~only if $N\varphi = N$, $G/N \in \mathcal{C}$, and~the~order of~$\varphi_{N}$ is finite and~is a~$\pi(\mathcal{C})$\nobreakdash-num\-ber.
\end{proof}

\begin{eproposition}\label{ep508}
Suppose that $\mathcal{C}$ is a~root class of~groups consisting only of~periodic groups and~closed under taking quotient groups, $G$ is an~abelian residually $\mathcal{C}$\nobreakdash-group, and~$H = G = K$. If the~order~$q$ of~the~automorphism~$\varphi$ is finite, then $\mathfrak{G}$ is residually a~$\mathcal{C}$\nobreakdash-group if and~only if $q$ is a~$\pi(\mathcal{C})$\nobreakdash-num\-ber.
\end{eproposition}

\begin{proof}
\textit{Necessity.} Because $q$ is the~order of~$\varphi$, for~any $i \in \{1,\, \ldots,\, q-1\}$, there exists an~element $g_{i} \in G$ such that $g_{i}\varphi^{i} \ne g_{i}$. Let
$$
S = \big\{g_{i}^{\vphantom{1}}\varphi^{i}_{\vphantom{1}}g_{i}^{-1} \mid 1 \leqslant i \leqslant q-1\big\}.
$$

Since $\mathfrak{G}$ is residually a~$\mathcal{C}$\nobreakdash-group, then, by~Proposition~\ref{ep301}, there exists a~subgroup $U \in\nolinebreak \mathcal{C}^{*}(\mathfrak{G})$ such that $U \cap S = \varnothing$. If $N = U \cap G$, then $N \in \mathcal{C}^{*}_{\cap}(G, H, K, \varphi)$ and,~by~Proposition~\ref{ep305}, $N$ is $\varphi$\nobreakdash-in\-var\-i\-ant, the~automorphism~$\varphi_{N}$ of~$G/N$ induced by~$\varphi$ has a~finite order~$q_{N}$, and~this order is a~$\pi(\mathcal{C})$\nobreakdash-num\-ber. It remains to~note that $N \cap S = \varnothing$ and~therefore $q_{N} = q$.
\smallskip

\textit{Sufficiency.} Suppose that $\Omega$ is the~family of~subgroups defined in~Proposition~\ref{ep507}, $M \in \mathcal{C}^{*}(G)$,~and
$$
N = \bigcap_{i=0}^{q-1}M\varphi^{i}.
$$
Then $N \leqslant M$, $N\varphi = N$, and,~by~Proposition~\ref{ep301}, $N \in \mathcal{C}^{*}(G)$. The~order of~the~automorphism~$\varphi_{N}$ induced by~$\varphi$ divides~$q$ and~therefore is a~$\pi(\mathcal{C})$\nobreakdash-num\-ber. Hence, $N \in \Omega$.

Thus, any subgroup from~$\mathcal{C}^{*}(G)$ contains a~subgroup from~$\Omega$ and,~since $G$ is residually a~$\mathcal{C}$\nobreakdash-group, $\Omega$ is a~filtration. Therefore, $\mathfrak{G}$ is residually a~$\mathcal{C}$\nobreakdash-group by~Proposition~\ref{ep507}. 
\end{proof}

\section{Proofs of~Theorems~\ref{et01}--\ref{et04} and~Corollaries~\ref{ec01}--\ref{ec04}}\label{esec06}

Throughout this section, it is assumed that $(G, H, K, \varphi)$ is an~HNN-tuple, $H$~and~$K$ lie in~the~center of~$G$, and~$\mathfrak{G} = \mathrm{HNN}(G, H, K, \varphi)$. It is also assumed that
\[
K_{0} = G,\ H_{1} = H,\ K_{1} = K,\ H_{i+1} = H_{i} \cap K_{i},\ K_{i+1} = H_{i+1}\varphi,\ P_{i} = H_{i}K_{i}\ \;(i \geqslant 1).
\]

\begin{eproposition}\label{ep601}
Suppose that $\mathcal{C}$ is a~root class of~groups and~$H \ne G \ne K$. If $\mathfrak{G}$ is residually a~$\mathcal{C}$\nobreakdash-group, then, for~any $n \geqslant 1$, the~HNN-extensions $\mathfrak{G}_{n-1} = \mathrm{HNN}(G, H_{n}, K_{n}, \varphi)$ and~$\mathfrak{K}_{n-1} = \mathrm{HNN}(K_{n-1}, H_{n}, K_{n}, \varphi)$ are also residually $\mathcal{C}$\nobreakdash-groups.
\end{eproposition}

\begin{proof}
Since $\mathfrak{G}$ is residually a~$\mathcal{C}$\nobreakdash-group, then the~family $\mathcal{C}_{\cap}^{*}(G, H, K, \varphi)$ is an~$(H, K)$\nobreakdash-fil\-tra\-tion by~Proposition~\ref{ep502}. Repeatedly using Proposition~\ref{ep402}, we see that $\mathcal{C}_{\cap}^{*}(G, H_{n}, K_{n}, \varphi)$ and~$\mathcal{C}_{\cap}^{*}(K_{n-1}, H_{n}, K_{n}, \varphi)$ are $(H_{n}, K_{n})$\nobreakdash-fil\-tra\-tions, and,~again by~Proposition~\ref{ep502}, $\mathfrak{G}_{n-1}$ and~$\mathfrak{K}_{n-1}$ are residually $\mathcal{C}$\nobreakdash-groups.
\end{proof}

\begin{eproposition}\label{ep602}
The~following statements hold.

\textup{1.}\hspace{1ex}Let $\pi$ be a~set of~primes. If $H_{n}$ and~$K_{n}$ are $\pi^{\prime}$\nobreakdash-iso\-lat\-ed in~$G$ for~some $n \geqslant 1$, then $H_{i}$ and~$K_{i}$ are $\pi^{\prime}$\nobreakdash-iso\-lat\-ed in~this group for~all $i \geqslant n$.

\textup{2.}\hspace{1ex}If a~number $n \geqslant 1$ and~a~subgroup $Q$ are such that $Q \leqslant H_{n} \cap K_{n}$ and~$Q\varphi = Q$, then $Q \leqslant H_{i} \cap K_{i}$ for~all $i \geqslant n$.

\textup{3.}\hspace{1ex}Let $\mathcal{C}$ be a~class of~groups closed under taking subgroups and~extensions. If $H \in \mathcal{C}^{*}(G)$ and~$K \in \mathcal{C}^{*}(G)$, then $H_{i} \in \mathcal{C}^{*}(G)$ and~$K_{i} \in \mathcal{C}^{*}(G)$ for~all $i \geqslant 1$.
\end{eproposition}

\begin{proof}
Let us use induction on~$i$ and~note that, for~all three statements, the~induction base is obvious.
\smallskip

1.\hspace{1ex}Suppose that $H_{i}$ and~$K_{i}$ are $\pi^{\prime}$\nobreakdash-iso\-lat\-ed in~$G$ for~some $i \geqslant n$. If an~element $g \in G$ and~a~number $q \in \pi^{\prime}$ are such that $g^{q} \in H_{i+1}$, then $g^{q} \in H_{i}$ and~$g^{q} \in K_{i}$. Since $H_{i}$ and~$K_{i}$ are $\pi^{\prime}$\nobreakdash-iso\-lat\-ed, then $g \in H_{i} \cap K_{i} = H_{i+1}$ and~therefore $H_{i+1}$ is $\pi^{\prime}$\nobreakdash-iso\-lat\-ed in~$G$. If, for~some $g \in G$ and~$q \in \pi^{\prime}$, the~inclusion $g^{q} \in K_{i+1}$ holds, then $g \in K_{i}$ because $K_{i}$ is $\pi^{\prime}$\nobreakdash-iso\-lat\-ed. Hence, the~element~$g\varphi^{-1}$ is defined and~$(g\varphi^{-1})^{q} \in H_{i+1}$. Since $H_{i+1}$ is $\pi^{\prime}$\nobreakdash-iso\-lat\-ed, it follows from~the~last inclusion that $g\varphi^{-1} \in H_{i+1}$ and~$g \in K_{i+1}$. Thus, $K_{i+1}$ is also $\pi^{\prime}$\nobreakdash-iso\-lat\-ed in~$G$.
\smallskip

2.\hspace{1ex}Let $Q \leqslant H_{i} \cap K_{i}$ for~some $i \geqslant n$. Then
$$
Q \leqslant H_{i+1} = H_{i} \cap K_{i},\quad
Q\varphi \leqslant H_{i+1}\varphi = K_{i+1},
$$
and~since $Q\varphi = Q$, then $Q \leqslant H_{i+1} \cap K_{i+1}$.
\smallskip

3.\hspace{1ex}Let $\mathcal{C}^{*}(G)$ contain $H_{i}$ and~$K_{i}$ for~some $i \geqslant 1$. Then $H_{i+1} = H_{i} \cap K_{i} \in \mathcal{C}^{*}(G)$ by~Proposition~\ref{ep301}. Since $\mathcal{C}$ is closed under taking subgroups, then $H_{i+1} \in \mathcal{C}^{*}(H_{i})$ and~it follows from~the~equalities $H_{i}\varphi = K_{i}$, $H_{i+1}\varphi = K_{i+1}$ that $K_{i+1} \in \mathcal{C}^{*}(K_{i})$. Hence, the~quotient group~$G/K_{i+1}$ is an~extension of~the~$\mathcal{C}$\nobreakdash-group~$K_{i}/K_{i+1}$ by~the~$\mathcal{C}$\nobreakdash-group~$G/K_{i}$, and~$K_{i+1} \in \mathcal{C}^{*}(G)$ because $\mathcal{C}$ is closed under taking extensions.
\end{proof}

\begin{eproposition}\label{ep603}
\textup{\cite[Proposition~5]{SokolovTumanova2016}}
Suppose that $\mathcal{C}$ is a~class of~groups consisting only of~periodic groups, $X$ is a~group, and~$Y$ is a~subgroup of~$X$. If $Y$ is $\mathcal{C}$\nobreakdash-sep\-a\-ra\-ble in~$X$, then it is $\pi(\mathcal{C})^{\prime}$\nobreakdash-iso\-lat\-ed in~this group. In~particular, if $X$ is residually a~$\mathcal{C}$\nobreakdash-group, then it has no $\pi(\mathcal{C})^{\prime}$\nobreakdash-tor\-sion.
\end{eproposition}

If $\pi$ is a~set of~primes, $X$ is a~group, and~$Y$ is a~subgroup of~$X$, then we denote by~$\mathcal{R}_{\pi^{\prime}}(X, Y)$ the~set of~elements of~$X$ defined as~follows: $x \in \mathcal{R}_{\pi^{\prime}}(X, Y)$ if and~only if $x^{q} \in Y$ for~some $\pi^{\prime}$\nobreakdash-num\-ber~$q$.

Obviously, $\mathcal{R}_{\pi^{\prime}}(X, Y) \subseteq \mathcal{I}_{\pi^{\prime}}(X, Y)$ and~the~equality $\mathcal{I}_{\pi^{\prime}}(X, Y) = \mathcal{R}_{\pi^{\prime}}(X, Y)$ holds if and~only if $\mathcal{R}_{\pi^{\prime}}(X, Y)$ is a~subgroup. It is easy to~see that, if $X$ is an~abelian group, then $\mathcal{R}_{\pi^{\prime}}(X, Y)$ is always a~subgroup.

\begin{eproposition}\label{ep604}
Suppose that $\mathcal{C}$ is a~class of~groups closed under taking quotient groups, $X$\kern-1pt{} is residually a~$\mathcal{C}$\nobreakdash-group, $Y$\kern-1pt{} is a~central subgroup of~$X$\kern-1pt{}.\kern-2pt{} Then the~$\pi(\mathcal{C})^{\prime}$\nobreakdash-iso\-la\-tor $\mathcal{I}_{\pi(\mathcal{C})^{\prime}}(X\kern-1pt{}, Y)$ is contained in~the~center of~$X$ and~coincides with~the~set $\mathcal{R}_{\pi(\mathcal{C})^{\prime}}(X, Y)$.
\end{eproposition}

\begin{proof}
Let $Z$ be the~center of~$X$. We show that $Z$ is $\mathcal{C}$\nobreakdash-sep\-a\-ra\-ble in~$X$.

If $x \in X \setminus Z$, then $[x, y] \ne 1$ for~some $y \in X$ and,~because $X$ is residually a~$\mathcal{C}$\nobreakdash-group, there exists a~subgroup $N \in \mathcal{C}^{*}(X)$ such that $[x, y] \notin N$. Hence, $[xN, yN] \ne N$ and~therefore $xN$ does not belong to~the~center $\mathcal{Z}(X/N)$ of~$X/N$. It is easy to~see that $ZN/N \leqslant \mathcal{Z}(X/N)$. Thus, $xN \notin ZN/N$ and~$x \notin ZN$. It remains to~note that $X/ZN \cong (X/N)/(ZN/N) \in \mathcal{C}$ because $\mathcal{C}$ is closed under taking quotient groups.

Now, it follows from~Proposition~\ref{ep603} that $Z$ is $\pi(\mathcal{C})^{\prime}$\nobreakdash-iso\-lat\-ed in~$X$, and~since $Y \leqslant Z$, then $\mathcal{I}_{\pi(\mathcal{C})^{\prime}}(X, Y) \leqslant Z$. Hence, $\mathcal{I}_{\pi(\mathcal{C})^{\prime}}(X, Y) = \mathcal{I}_{\pi(\mathcal{C})^{\prime}}(Z, Y)$, $\mathcal{R}_{\pi(\mathcal{C})^{\prime}}(X, Y) = \mathcal{R}_{\pi(\mathcal{C})^{\prime}}(Z, Y)$, and~$\mathcal{I}_{\pi(\mathcal{C})^{\prime}}(Z, Y) = \mathcal{R}_{\pi(\mathcal{C})^{\prime}}(Z, Y)$ by~the~above remark.
\end{proof}

\begin{eproposition}\label{ep605}
\textup{\cite[Proposition~3]{SokolovTumanova2020IVM}}
Suppose that $\mathcal{C}$ is a~class of~groups closed under taking quotient groups, $X$ is a~group, and~$Y$ is a~normal subgroup of~$X$. Then $Y$ is $\mathcal{C}$\nobreakdash-sep\-a\-ra\-ble in~$X$ if and~only if $X/Y$ is residually a~$\mathcal{C}$\nobreakdash-group.
\end{eproposition}

\begin{proof}[\textup{\textbf{Proof of~Theorem~\ref{et03}}}]
\textit{Necessity}. Since $\mathfrak{G}$ is residually a~$\mathcal{C}$\nobreakdash-group and~$\mathcal{C}$ is closed under taking subgroups, then $E$ is residually a~$\mathcal{C}$\nobreakdash-group. By~Proposition~\ref{ep601}, the~HNN-extension $\mathfrak{G}_{n-1} = \mathrm{HNN}(G, H_{n}, K_{n}, \varphi \rangle$ is also residually a~$\mathcal{C}$\nobreakdash-group. Since $H_{n} = H_{n+1} = H_{n} \cap K_{n}$, then $H_{n} \leqslant K_{n}$, and~we get the~equality $H_{n} = K_{n}$ by~applying Proposition~\ref{ep503} to~$\mathfrak{G}_{n-1}$.{\parfillskip=0pt\par}

It follows from~Proposition~\ref{ep502} that $\mathcal{C}_{\cap}^{*}(G, H, K, \varphi)$ is an~$(H, K)$-fil\-tra\-tion. By~Proposition~\ref{ep304}, $\mathcal{C}_{\cap}^{*}(G, H, K, \varphi) \subseteq \mathcal{C}^{*}(G)$, and~therefore $\mathcal{C}^{*}(G)$ is also an~$(H, K)$\nobreakdash-fil\-tra\-tion. Hence, $H$~and~$K$~are $\mathcal{C}$\nobreakdash-sep\-a\-ra\-ble in~$G$ and~$\pi(\mathcal{C})^{\prime}$\nobreakdash-iso\-lat\-ed in~this group by~Proposition~\ref{ep603}.
\smallskip

\textit{Sufficiency}. Since
$$
H_{n+1} = H_{n} = K_{n},\quad K_{n+1} = H_{n+1}\varphi = H_{n}\varphi = K_{n},
$$
then $\mathfrak{K}_{n} = \mathrm{HNN}(K_{n}, H_{n+1}, K_{n+1}, \varphi)$ is a~split extension of~$K_{n}$ by~the~infinite cyclic group generated by~$t$. Therefore, the~mapping of~the~generators of~$\mathfrak{K}_{n}$ to~the~corresponding elements of~$\mathfrak{G}$ can be extended to~an~injective homomorphism taking $\mathfrak{K}_{n}$ onto~$E$. Hence, $\mathfrak{K}_{n}$ is residually a~$\mathcal{C}$\nobreakdash-group, and~$\mathcal{C}_{\cap}^{*}(K_{n}, H_{n+1}, K_{n+1}, \varphi)$ is a~filtration by~Proposition~\ref{ep502}. Because $H_{n+1} = K_{n} = K_{n+1}$, this family is actually an~$(H_{n+1}, K_{n+1})$-fil\-tra\-tion. By~Proposition~\ref{ep302}, it is closed under taking finite intersections and~therefore is a~strong $(H_{n+1}, K_{n+1})$-fil\-tra\-tion. It follows from~Proposition~\ref{ep305} that $\mathcal{C}_{\cap}^{*}(K_{n}^{\vphantom{*}}, H_{n+1}^{\vphantom{*}}, K_{n+1}^{\vphantom{*}}, \varphi) = \mathcal{C}_{n}^{*}(K_{n}^{\vphantom{*}}, H_{n+1}^{\vphantom{*}}, K_{n+1}^{\vphantom{*}}, \varphi)$. Hence, $\mathcal{C}_{n}^{*}(K_{n}^{\vphantom{*}}, H_{n+1}^{\vphantom{*}}, K_{n+1}^{\vphantom{*}}, \varphi)$ is also a~strong $(H_{n+1}^{\vphantom{*}}, K_{n+1}^{\vphantom{*}})$-fil\-tra\-tion.

Since $G$ is residually a~$\mathcal{C}$\nobreakdash-group and~$\mathcal{C}$ is closed under taking subgroups, then $K_{i}$ ($i \geqslant\nolinebreak 0$) are also residually $\mathcal{C}$\nobreakdash-groups. For~any $i \in \{0,\, 1,\, \ldots,\, n-1\}$, $P_{i+1}$ is central in~$K_{i}$. Hence, by~Proposition~\ref{ep604}, the~$\pi(\mathcal{C})^{\prime}$\nobreakdash-iso\-la\-tor $X_{i} = \mathcal{I}_{\pi(\mathcal{C})^{\prime}}(K_{i}, P_{i+1})$ lies in~the~center of~$K_{i}$ and~coincides with~the~set $\mathcal{R}_{\pi(\mathcal{C})^{\prime}}(K_{i}, P_{i+1})$. Therefore, $X_{i}/P_{i+1}$ is a~periodic $\pi(\mathcal{C})^{\prime}$\nobreakdash-group.

By~the~hypothesis of~the~theorem, $X_{i}$ is $\mathcal{C}$\nobreakdash-sep\-a\-ra\-ble in~$K_{i}$ for~any $i \in \{0,\, 1,\, \ldots,\, n-1\}$, and~therefore $K_{i}/X_{i}$ is residually a~$\mathcal{C}$\nobreakdash-group by~Proposition~\ref{ep605}. Again by~the~hypothesis, $K_{i}$ is $\mathcal{C}$\nobreakdash-reg\-u\-lar with~respect to~$P_{i+1}$. By~Proposition~\ref{ep602}, $H_{i+1}$ and~$K_{i+1}$ are $\pi(\mathcal{C})^{\prime}$\nobreakdash-iso\-lat\-ed in~$G$ and~therefore in~$K_{i}$. Hence, we can successively apply Proposition~\ref{ep404} to~the~HNN-tuples $(K_{i}, H_{i+1}, K_{i+1}, \varphi)$, $i = n-1,\, \ldots,\, 1,\, 0$. As~a~result, we see that $\mathcal{C}_{0}^{*}(G, H, K, \varphi)$ is a~strong $(H, K)$-fil\-tra\-tion. By~Proposition~\ref{ep304}, $\mathcal{C}_{0}^{*}(G, H, K, \varphi) \subseteq \mathcal{C}_{\cap}^{*}(G, H, K, \varphi)$. Therefore, $\mathcal{C}_{\cap}^{*}(G, H, K, \varphi)$ is an~$(H, K)$-fil\-tra\-tion, and~$\mathfrak{G}$ is residually a~$\mathcal{C}$\nobreakdash-group by~Proposition~\ref{ep502}.
\end{proof}

\begin{proof}[\textup{\textbf{Proof of~Theorem~\ref{et04}}}]
Since $G$ is residually a~$\mathcal{C}$\nobreakdash-group and~$\mathcal{C}$ is closed under taking subgroups, then $K_{n}$ is also residually a~$\mathcal{C}$\nobreakdash-group and,~by~Proposition~\ref{ep505}, $\mathcal{C}^{*}(K_{n}, H_{n+1}, K_{n+1}, \varphi)$ is an~$(H_{n+1}, K_{n+1})$-fil\-tra\-tion. By~Proposition~\ref{ep302}, this family is closed under taking finite intersections and~therefore turns~out to~be a~strong $(H_{n+1}, K_{n+1})$-fil\-tra\-tion. By~the~hypothesis of~the~theorem, $P_{i+1}$ is $\mathcal{C}$\nobreakdash-sep\-a\-ra\-ble in~$K_{i}$ for~each $i \in \{0,\, 1,\, \ldots,\, n-1\}$.\linebreak Hence, $K_{i}/P_{i+1}$ is residually a~$\mathcal{C}$\nobreakdash-group by~Proposition~\ref{ep605}. If we take the~set of~all primes as~$\pi$ and~successively apply Proposition~\ref{ep404} to~the~HNN-tuples $(K_{i}, H_{i+1}, K_{i+1}, \varphi)$ and~the~subgroups $X_{i+1} = P_{i+1}$ ($i = n-1,\, \ldots,\, 1,\, 0$), then we see that $\mathcal{C}^{*}(G, H, K, \varphi)$ is a~strong $(H, K)$-fil\-tra\-tion. Since $\mathcal{C}^{*}(G, H, K, \varphi) = \mathcal{C}_{\cap}^{*}(G, H, K, \varphi)$ by~Proposition~\ref{ep304}, then $\mathcal{C}_{\cap}^{*}(G, H, K, \varphi)$ is an~$(H, K)$-fil\-tra\-tion. Thus, $\mathfrak{G}$ is residually a~$\mathcal{C}$\nobreakdash-group by~Proposition~\ref{ep502}.
\end{proof}

\begin{proof}[\textup{\textbf{Proof of~Theorem~\ref{et01}}}]
Statement~I follows from~Theorem~\ref{et04}. Let us prove Statement~II. 

We take a~number $i \geqslant 0$ and~show that $K_{i}/P_{i+1} \in \mathcal{C}$.

If $Q \leqslant H \cap K$ and~$Q\varphi = Q$, then $Q \leqslant H_{i+1} \cap K_{i+1}$ by~Proposition~\ref{ep602} and~$Q \leqslant P_{i+1}$. Hence, $G/P_{i+1} \cong (G/Q)/(P_{i+1}/Q) \in \mathcal{C}$ and~$K_{i}/P_{i+1} \in \mathcal{C}$ because $\mathcal{C}$ is closed under taking subgroups and~quotient groups.
If $H \cap Q = 1 = K \cap Q$, then $K_{i} \cap Q = 1$ and~$K_{i}/P_{i+1} \cong K_{i}/P_{i+1}(K_{i} \cap Q) \cong K_{i}Q/P_{i+1}Q$. Since $Q \leqslant P_{i+1}Q$, then, as~above, $K_{i}Q/P_{i+1}Q \in \mathcal{C}$.

Thus, $K_{i}/P_{i+1} \in \mathcal{C}$. It follows that $P_{i+1}$ is $\mathcal{C}$\nobreakdash-sep\-a\-ra\-ble in~$K_{i}$ and,~by~Proposition~\ref{ep603}, is $\pi(\mathcal{C})^{\prime}$\nobreakdash-iso\-lat\-ed in~this group. Therefore, $P_{i+1} = \mathcal{I}_{\pi(\mathcal{C})^{\prime}}(K_{i}, P_{i+1})$. We note also that if $N \in \mathcal{C}^{*}(P_{i+1})$, then $K_{i}/N$ is an~extension of~the~$\mathcal{C}$\nobreakdash-group $P_{i+1}/N$ by~the~$\mathcal{C}$\nobreakdash-group $K_{i}/P_{i+1}$ and~$K_{i}/N \in \mathcal{C}$ because $\mathcal{C}$ is closed under taking extensions. Hence, $\mathcal{C}^{*}(P_{i+1}) \subseteq \mathcal{C}^{*}(K_{i})$ and~therefore $K_{i}$ is $\mathcal{C}$\nobreakdash-reg\-u\-lar with~respect to~$P_{i+1}$.

Thus, the~statement to~be proved follows from~Theorem~\ref{et03}. We only need to~note that $H$~and~$K$ are $\mathcal{C}$\nobreakdash-sep\-a\-ra\-ble in~$G$ by~Proposition~\ref{ep504} and~are $\pi(\mathcal{C})^{\prime}$\nobreakdash-iso\-lat\-ed in~this group by~Proposition~\ref{ep603}.
\end{proof}

\begin{proof}[\textup{\textbf{Proof of~Corollary~\ref{ec01}}}]
Because $H$ is finite, the~equalities $H_{n} = H_{n+1} = H_{n} \cap K_{n}$ hold for~some $n \geqslant 1$. Therefore, $H_{n} \leqslant K_{n}$. But~the~subgroups~$H_{n}$ and~$K_{n}$ are finite and~isomorphic, hence, $H_{n} = K_{n}$. Since $G$ is residually a~$\mathcal{C}$\nobreakdash-group, then, by~Proposition~\ref{ep301}, there exists a~subgroup $Q \in \mathcal{C}^{*}(G)$ such that $H \cap Q = 1 = K \cap Q$. Therefore, Statement~I follows from~Theorem~\ref{et01}. Let us prove Statement~II. 

Since $H \cap Q = 1$, then $H_{n}$ is embedded into~the~$\mathcal{C}$\nobreakdash-group $G/Q$ and~belongs to~$\mathcal{C}$ because this class is closed under taking subgroups. The~restriction of~$\varphi$ to~the~finite subgroup~$H_{n}$ has a~finite order~$q$. Hence, by~Proposition~\ref{ep508}, the~subgroup $E = \operatorname{sgp}\{H_{n}, t\}$ of~$\mathfrak{G}$ is residually a~$\mathcal{C}$\nobreakdash-group if and~only if $q$ is a~$\pi(\mathcal{C})$\nobreakdash-num\-ber. If $H = G = K$, then $\mathfrak{G} = E$, and~the~required statement is proved. Otherwise, the~inequalities $G \ne H$ and~$G \ne K$ hold simultaneously because $H$ and~$K$ are finite and~isomorphic. Therefore, the~statement of~the~corollary follows from~Theorem~\ref{et01}.
\end{proof}

\begin{proof}[\textup{\textbf{Proof of~Corollary~\ref{ec02}}}]
I.\hspace{1ex}The~conditions of~the~statement are necessary for~$\mathfrak{G}$ to~be residually a~$\mathcal{C}$\nobreakdash-group by~Proposition~\ref{ep506} and~because $\mathcal{C}$ is closed under taking subgroups. Their sufficiency follows from~Theorem~\ref{et01}.
\smallskip

II.\hspace{1ex}\textit{Necessity.} By~Propositions~\ref{ep506} and~\ref{ep602}, there exists a~subgroup $Q \in \mathcal{C}^{*}(G)$ such that $Q\varphi = Q$ and~$Q \leqslant H_{i} \cap K_{i}$ for~all $i \geqslant 1$. Since $\mathcal{C}$ consists of~finite groups, $Q$ has a~finite index in~$G$ and~therefore $H_{n} = H_{n+1}$ for~some~$n$. Because $\mathcal{C}$ is closed under taking subgroups, $G$ is residually a~$\mathcal{C}$\nobreakdash-group. Therefore, Conditions~2 and~3 follow from~Theorem~\ref{et01}. Since $\mathcal{C}$ is closed under taking quotient groups, Condition~1 follows from~the~relations
$$
G/H \cong (G/Q)/(H/Q),\quad G/K \cong (G/Q)/(K/Q).
$$

\textit{Sufficiency.} By~Proposition~\ref{ep602}, $H_{i} \in \mathcal{C}^{*}(G)$ and~$K_{i} \in \mathcal{C}^{*}(G)$ for~any $i \geqslant 1$. Hence, $H_{n} \in \mathcal{C}^{*}(G)$, $H_{n} \leqslant H \cap K$, and~$H_{n}\varphi = H_{n}$. Since $E$ is residually a~$\mathcal{C}$\nobreakdash-group and~$\mathcal{C}$ is closed under taking subgroups, then $H_{n}$ is also residually a~$\mathcal{C}$\nobreakdash-group. Therefore, $G$ is an~extension of~the~residually $\mathcal{C}$\nobreakdash-group~$H_{n}$ by~the~$\mathcal{C}$\nobreakdash-group~$G/H_{n}$. Such an~extension is residually a~$\mathcal{C}$\nobreakdash-group by~Proposition~\ref{ep501}. Thus, it follows from~Theorem~\ref{et01} that $\mathfrak{G}$ is residually a~$\mathcal{C}$\nobreakdash-group.
\end{proof}

\begin{eproposition}\label{ep606}
\textup{\cite[Proposition~18]{SokolovTumanova2020IVM}}
If $\mathcal{C}$ is a~root class of~groups consisting only of~periodic groups and~closed under taking quotient groups, then any $\pi(\mathcal{C})$\nobreakdash-bound\-ed solvable $\mathcal{C}$\nobreakdash-group is finite.
\end{eproposition}

\begin{eproposition}\label{ep607}
\textup{\cite[Propositions~1,~2,~3]{Sokolov2014}}
If $\pi$ is a~non-empty set of~primes, then the~following statements hold.

\textup{1.}\hspace{1ex}Any $\pi$\nobreakdash-bound\-ed abelian group is of~finite rank.

\textup{2.}\hspace{1ex}The~classes of~$\pi$\nobreakdash-bound\-ed abelian, $\pi$\nobreakdash-bound\-ed nilpotent, and~$\pi$\nobreakdash-bound\-ed solvable groups are closed under taking subgroups and~quotient groups.

\textup{3.}\hspace{1ex}If a~$\pi$\nobreakdash-bound\-ed solvable group is abelian, then it belongs to~the~class of~$\pi$\nobreakdash-bound\-ed abelian groups.
\end{eproposition}

\begin{proof}[\textup{\textbf{Proof of~Theorem~\ref{et02}}}]
Let us show that, for~any $i \in \{0,\, 1,\, \ldots,\, n-1\}$, $K_{i}$ is $\mathcal{C}$\nobreakdash-reg\-u\-lar with~respect to~$P_{i+1}$. Then the~required statement will follow from~Theorem~\ref{et03}.

Suppose that $i \in \{0,\, 1,\, \ldots,\, n-1\}$ and~$M \in \mathcal{C}^{*}(P_{i+1})$. We need to~find a~subgroup $N \in \mathcal{C}^{*}(K_{i})$ such that $N \cap P_{i+1} = M$.

Let us denote the~$\pi(\mathcal{C})^{\prime}$\nobreakdash-iso\-la\-tor $\mathcal{I}_{\pi(\mathcal{C})^{\prime}}(K_{i}, M)$ by~$\mathfrak{I}$ for~brevity. Since $G$ is residually a~$\mathcal{C}$\nobreakdash-group and~$\mathcal{C}$ is closed under taking subgroups, then $K_{i}$ is also residually a~$\mathcal{C}$\nobreakdash-group. Therefore, by~Proposition~\ref{ep604}, $\mathfrak{I}$ lies in~the~center of~$K_{i}$ and~coincides with~the~set $\mathcal{R}_{\pi(\mathcal{C})^{\prime}}(K_{i}, M)$. It follows from~the~equality $\mathfrak{I} = \mathcal{R}_{\pi(\mathcal{C})^{\prime}}(K_{i}, M)$ that, if $x \in \mathfrak{I} \cap P_{i+1}$, then $x^{q} \in M$ for~some $\pi(\mathcal{C})^{\prime}$\nobreakdash-num\-ber~$q$. But~the~quotient group~$P_{i+1}/M$ belongs to~$\mathcal{C}$ and~therefore has no $\pi(\mathcal{C})^{\prime}$\nobreakdash-tor\-sion. Hence, $xM = 1$ and~$x \in M$. Thus, $\mathfrak{I} \cap P_{i+1} \leqslant M$ and,~because the~opposite inclusion is obvious, $\mathfrak{I} \cap P_{i+1} = M$.

Since $H$ and~$K$ are $\pi(\mathcal{C})$\nobreakdash-bound\-ed, then it follows from~Propositions~\ref{ep606} and~\ref{ep607} that the~$\mathcal{C}$\nobreakdash-group~$P_{i+1}/M$ is $\pi(\mathcal{C})$\nobreakdash-bound\-ed and~therefore finite. Hence, the~subgroup
$$
P_{i+1}\mathfrak{I}/\mathfrak{I} \cong P_{i+1}/(P_{i+1} \cap \mathfrak{I}) = P_{i+1}/M
$$
is also finite. By~the~hypothesis of~the~theorem, $\mathfrak{I}$ is $\mathcal{C}$\nobreakdash-sep\-a\-ra\-ble in~$K_{i}$. Therefore, by~Proposition~\ref{ep605}, $K_{i}/\mathfrak{I}$ is residually a~$\mathcal{C}$\nobreakdash-group and,~by~Proposition~\ref{ep301}, there exists a~subgroup $N/\mathfrak{I} \in \mathcal{C}^{*}(K_{i}/\mathfrak{I})$ such that $N/\mathfrak{I} \cap P_{i+1}\mathfrak{I}/\mathfrak{I} = 1$. Then $N \in \mathcal{C}^{*}(K_{i})$ and,~as~it is easy to~see, $N \cap P_{i+1} = M$. Thus, $N$ is the~required subgroup.
\end{proof}

\begin{eproposition}\label{ep608}
Suppose that $\mathcal{C}$ is a~root class of~groups consisting only of~periodic groups and~closed under taking quotient groups, $H \ne G \ne K$, and~there exists $m \geqslant 0$ such that at~least one of~the~following conditions holds:

$(\alpha)$\hspace{1ex}$H_{m+1}$ and~$K_{m+1}$ are finitely generated;

$(\beta)$\hspace{1ex}$H_{m+1}$ and~$K_{m+1}$ are $\pi(\mathcal{C})$\nobreakdash-bound\-ed and~$\pi^{\prime}$\nobreakdash-iso\-lat\-ed in~$K_{m}$ for~some finite subset $\pi$ of~$\pi(\mathcal{C})$.

If $\mathfrak{G}$ is residually a~$\mathcal{C}$\nobreakdash-group, then $H_{n} = H_{n+1}$ for~some~$n$.
\end{eproposition}

\begin{proof}
Because $\mathcal{C}$ contains non-triv\-i\-al groups, $\pi(\mathcal{C})$ is non-empty. Hence, the~rank $\operatorname{rk}H_{m+1}$ of~$H_{m+1}$ is finite by~Proposition~\ref{ep607}, and~since $\operatorname{rk}H_{i+1} \leqslant \operatorname{rk}H_{i}$ for~any $i \geqslant 1$, then $\operatorname{rk}H_{l} = \operatorname{rk}H_{l+1}$ for~some $l \geqslant m+1$. It follows from~the~relations $H_{l} \cong K_{l}$ and~$H_{l+1} \cong K_{l+1}$ that
$$
\operatorname{rk}K_{l} = \operatorname{rk}H_{l} = \operatorname{rk}H_{l+1} = \operatorname{rk}K_{l+1}
$$
and~therefore the~quotient groups~$K_{l}/H_{l+1}$ and~$K_{l}/K_{l+1}$ are periodic. If $H_{m+1}$ and~$K_{m+1}$ are finitely generated, then $K_{l}/H_{l+1}$ and~$K_{l}/K_{l+1}$ are finite. Let us show that this is also true if $(\beta)$ holds.

Since $H_{m+1}$ and~$K_{m+1}$ are $\pi^{\prime}$\nobreakdash-iso\-lat\-ed in~$K_{m}$, then, by~Proposition~\ref{ep602} applied to the~HNN-tuple $(K_{m}, H_{m+1}, K_{m+1}, \varphi)$, $H_{l+1}$ and~$K_{l+1}$ are also $\pi^{\prime}$\nobreakdash-iso\-lat\-ed in~this group. Hence, $K_{l}/H_{l+1}$ and~$K_{l}/K_{l+1}$ are periodic $\pi$\nobreakdash-groups, and~each of~them has a~finite number of~primary components because $\pi$ is finite. By~Proposition~\ref{ep607}, these groups are $\pi(\mathcal{C})$\nobreakdash-bound\-ed, and~since $\pi \subseteq \pi(\mathcal{C})$, all their primary components are finite. Thus, $K_{l}/H_{l+1}$ and~$K_{l}/K_{l+1}$ are also finite.

By~Proposition~\ref{ep601}, the~HNN-ex\-ten\-sions
$$
\mathfrak{G}_{l} = \mathrm{HNN}(G, H_{l+1}, K_{l+1}, \varphi),\quad
\mathfrak{K}_{l} = \mathrm{HNN}(K_{l}, H_{l+1}, K_{l+1}, \varphi)
$$
are residually $\mathcal{C}$\nobreakdash-groups. If $H_{l+1} = K_{l}$ or~$K_{l+1} = K_{l}$, then $H_{l+1} \geqslant K_{l+1}$ or~$H_{l+1} \leqslant K_{l+1}$ and~it follows from~Proposition~\ref{ep503} applied to~$\mathfrak{G}_{l}$ that $H_{l+1} = K_{l+1}$. Hence, we can put $n = l+1$. Let $H_{l+1} \ne K_{l} \ne K_{l+1}$. Then, by~Proposition~\ref{ep506} applied to~$\mathfrak{K}_{l}$, there exists a~subgroup $Q \in \mathcal{C}^{*}(K_{l})$ such that $Q \leqslant H_{l+1} \cap K_{l+1}$ and~$Q\varphi = Q$. The~$\mathcal{C}$\nobreakdash-group~$K_{l}/Q$ is finite: this is obvious if $(\alpha)$ holds, and~follows from~Propositions~\ref{ep606},~\ref{ep607} if $(\beta)$ takes place. By~Proposition~\ref{ep602}, $Q \leqslant H_{i}$ for~any $i \geqslant l+1$. Thus, there exists $n \geqslant l+1$ such that $H_{n} = H_{n+1}$.
\end{proof}

\begin{eproposition}\label{ep609}
\textup{\cite[Proposition~10]{Sokolov2017}}
If $\mathcal{C}$ is a~root class of~groups consisting only of~periodic groups, then any finite solvable $\pi(\mathcal{C})$\nobreakdash-group belongs to~$\mathcal{C}$.
\end{eproposition}

\begin{eproposition}\label{ep610}
If $\mathcal{C}$ is a~root class of~groups consisting only of~periodic groups and~$\pi(\mathcal{C})$ contains all prime numbers, then all the~subgroups of~an~arbitrary $\pi(\mathcal{C})$\nobreakdash-bound\-ed solvable group are $\mathcal{C}$\nobreakdash-sep\-a\-ra\-ble.
\end{eproposition}

\begin{proof}
Let $X$ be a~$\pi(\mathcal{C})$\nobreakdash-bound\-ed solvable group, and~let $Y$ be a~subgroup of~$X$. By~Theorem~6 from~\cite{Malcev1958}, $Y$ is $\mathcal{F}$\nobreakdash-sep\-a\-ra\-ble in~$X$, where $\mathcal{F}$ is the~class of~all finite groups. Since any homomorphic image of~$X$ is a~solvable group, then $Y$ turns~out to~be $\mathcal{F}\mathcal{S}$\nobreakdash-sep\-a\-ra\-ble in~$X$, where $\mathcal{FS}$ is the~class of~all finite solvable groups. By~Proposition~\ref{ep609}, $\mathcal{FS} \subseteq \mathcal{C}$. Hence, $Y$ is $\mathcal{C}$\nobreakdash-sep\-a\-ra\-ble.
\end{proof}

\begin{eproposition}\label{ep611}
\textup{\cite[Proposition~8]{SokolovTumanova2016}}
If $\mathcal{C}$ is a~root class of~groups consisting only of~periodic groups and~$X$ is a~$\pi(\mathcal{C})$\nobreakdash-bound\-ed nilpotent group, then every $\pi(\mathcal{C})^{\prime}$\nobreakdash-iso\-lat\-ed subgroup of~$X$ is $\mathcal{C}$\nobreakdash-sep\-a\-ra\-ble in~this group.
\end{eproposition}

\begin{proof}[\textup{\textbf{Proof of~Corollaries~\ref{ec03} and~\ref{ec04}}}]
Let us verify that the~conditions of~Theorem~\ref{et02} hold.

By~Proposition~\ref{ep607}, $H_{i}$ and~$K_{i}$ ($i \geqslant 1$) are $\pi(\mathcal{C})$\nobreakdash-bound\-ed abelian groups. Therefore, it follows from~Propositions~\ref{ep610} and~\ref{ep611} that all the~$\pi(\mathcal{C})^{\prime}$\nobreakdash-iso\-lat\-ed subgroups of~$K_{i}$~($i \geqslant\nolinebreak 0$) are $\mathcal{C}$\nobreakdash-sep\-a\-ra\-ble in~this group. In~particular, if $N \in \mathcal{C}^{*}(P_{i+1})$, then the~subgroup $\mathcal{I}_{\pi(\mathcal{C})^{\prime}}(K_{i}, N)$ is $\mathcal{C}$\nobreakdash-sep\-a\-ra\-ble in~$K_{i}$.

The~fact that $G$ is residually a~$\mathcal{C}$\nobreakdash-group follows from

--\hspace{1ex}Proposition~\ref{ep610} if $\pi(\mathcal{C})$ contains all primes and~$G$ is a~$\pi(\mathcal{C})$\nobreakdash-bound\-ed solvable group;

--\hspace{1ex}Proposition~\ref{ep611} and~the~absence of~$\pi(\mathcal{C})^{\prime}$\nobreakdash-tor\-sion in~$G$ if the~last group is $\pi(\mathcal{C})$\nobreakdash-bound\-ed nilpotent;

--\hspace{1ex}the~assumption that $\mathcal{C}$ is closed under taking subgroups if $\mathfrak{G}$ is residually a~$\mathcal{C}$\nobreakdash-group.

Finally, $H_{n} = H_{n+1}$ for~some $n \geqslant 1$: it is obvious if $H_{n} = K_{n}$, and~is guaranteed by~Proposition~\ref{ep608} if $\mathfrak{G}$ is residually a~$\mathcal{C}$\nobreakdash-group. Thus, the~necessity of~the~statements of~both corollaries follows from~Theorem~\ref{et02}. We only need to~note that if $\mathfrak{G}$ is residually a~$\mathcal{C}$\nobreakdash-group, then $G$ is also residually a~$\mathcal{C}$\nobreakdash-group and,~by~Proposition~\ref{ep603}, has no $\pi(\mathcal{C})^{\prime}$\nobreakdash-tor\-sion.

The~group~$E$ is $\pi(\mathcal{C})$\nobreakdash-bound\-ed solvable as~an~extension of~the~$\pi(\mathcal{C})$\nobreakdash-bound\-ed abelian group~$H_{n}$ by~the~infinite cyclic group~$\langle t \rangle$, which is also $\pi(\mathcal{C})$\nobreakdash-bound\-ed abelian. If $\pi(\mathcal{C})$ contains all primes, then $H$ and~$K$ are $\pi(\mathcal{C})^{\prime}$\nobreakdash-iso\-lat\-ed in~$G$ and,~by~Proposition~\ref{ep610}, $E$ is residually a~$\mathcal{C}$\nobreakdash-group. Therefore, the~sufficiency of~the~statements of~the~corollaries also follows from~Theorem~\ref{et02}.
\end{proof}

\end{document}